\documentclass[eqno] {amsart}
\usepackage{amscd}
\usepackage{amssymb}
\usepackage{cite}
\usepackage{amsmath}
\usepackage{a4wide}

\newtheorem{theorem}{Theorem}[section]
\newtheorem{lemma}[theorem]{Lemma}
\usepackage{comment}
\theoremstyle{theorem}
\newtheorem{definition}[theorem]{Definition}
\newtheorem{example}[theorem]{Example}
\newtheorem{question}[theorem]{Question}

\usepackage{bookmark}

\newtheorem{lem}[theorem]{\sc \bf Lemma}

\newtheorem{cor}[theorem]{\sc \bf Corollary}

\newtheorem{proposition}[theorem]{\sc \bf Proposition}

\newtheorem{rem}[theorem]{\sc \bf Remark}
\newtheorem{remark}[theorem]{\sc \bf Remark}
 \usepackage{color}
\usepackage{hyperref}					
\hypersetup{colorlinks,
	linkcolor=blue,%
	citecolor=blue}

\newcommand{\cA}{\mathcal{A}}

\newcommand{\cH}{\mathcal{H}}
\newcommand{\cF}{\mathcal{F}}

\newcommand{\cM}{\mathcal{M}}
\newcommand{\cP}{\mathcal{P}}

\newcommand{\norm}[1]{\left\lVert#1\right\rVert}

\numberwithin{equation}{section}

\begin{document}

\title[Hermitian operators and isometries on   symmetric operator spaces]{Hermitian operators and isometries on symmetric operator spaces}


\author{Jinghao Huang}
\address{Institute of Advanced Study of Mathematics, HIT, Harbin, 150001, China}
\curraddr{}
\email{jinghao.huang@hit.edu.cn}

\author{Fedor Sukochev}
\address{School of Mathematics and Statistics, University of New South Wales, Kensington, 2052, NSW, Australia }
\curraddr{}
\email{f.sukochev@unsw.edu.au}
\thanks{Jinghao Huang was supported by the NNSF of China.  Fedor Sukochev was  supported by  the Australian Research Council  (FL170100052)}

\date{}



\dedicatory{ }

\begin{abstract}
Let $\cM$ be an  atomless semifinite von Neumann algebra (or an atomic  von Neumann algebra with all atoms having the same trace) acting on a (not necessarily separable)  Hilbert space $\cH$ equipped with a semifinite faithful normal trace $\tau$.
Let $E(\cM,\tau) $  be a  symmetric operator space affiliated with $ \cM $, whose norm is order continuous and  is not proportional to the Hilbertian norm $\norm{\cdot}_2$ on $L_2(\cM,\tau)$.
We obtain  a  general   description of all bounded hermitian operators on $E(\cM,\tau)$.
This is the first time that the description of  hermitian operators on  a symmetric operator space (even for a noncommutative $L_p$-space) is obtained  in the setting of    general (non-hyperfinite) von Neumann algebras.
 As an application, we resolve  a long-standing open problem concerning  the description of isometries raised in  the 1980s, which
generalizes and unifies  numerous  earlier  results.
\end{abstract}

\maketitle

\section{Introduction}
The main purpose of this paper is to answer the following   long-standing  open question concerning  isometries on a symmetric operator space  (see e.g. \cite{HSZ,CMS,AC,Sukochev})
  \begin{question}\label{mainq}
   If $E (0,\infty)$ is a separable symmetric function on $(0,\infty)$ and if $(\cM,\tau)$ is a semifinite von Neumann algebra (on a separable Hilbert space) with a semifinite faithful normal trace $\tau$,
   then how can one describe the family of    surjective  isometries on  the  symmetric operator space $E(\cM,\tau)$ associated with $E(0,\infty)$?
  \end{question}
This is one of the most fundamental questions in the theory of symmetric operator/function/sequence spaces, and it  has attracted a substantial  amount of interest.

The study of  the above question has a very long history, initiated by Stefan Banach~\cite{Banach},  who obtained the general form of isometries between  $L_p$-spaces on a finite measure space   in the 1930s.
This result was    extended by Lamperti~\cite{Lamperti} to certain Orlicz function spaces over  $\sigma$-finite measure spaces.
Representations of isometries between more general complex symmetric function spaces were later obtained by Lumer and by Zaidenberg~\cite{Zaidenberg,Z77,FJ} (see Arazy's paper \cite{Arazy85} for the case of complex sequence spaces).
Precisely, Zaidenberg showed that under mild conditions on the complex  function spaces $E_1(\Omega_1, \Sigma_1,\mu_1) $ and  $E_2(\Omega_2, \Sigma_2,\mu_2)$   over the
  atomless $\sigma$-finite measure spaces $ (\Omega_1, \Sigma_1,\mu_1) $ and  $ (\Omega_2, \Sigma_2,\mu_2)$,  any surjective  isometry $T$ between two complex symmetric function spaces $E_1(\Omega_1, \Sigma_1,\mu_1) $ and  $E_2(\Omega_2, \Sigma_2,\mu_2)$ must be of the   elementary form
\begin{align}\label{formf}
(Tf)(t)=h(t)(T_1f)(t), ~f\in E_1,
\end{align}where $T_1$ is the operator induced by a regular set isomorphism from $ \Omega_1 $ onto $\Omega_2$ and $h$ is a measurable function on $\Omega_2$\cite[Theorem 5.3.5]{FJ} (see also \cite{Z77,Zaidenberg}).
Let $(\Gamma, \Sigma, \mu)$ be a discrete measure space on a set $\Gamma$ with $\mu(\{\gamma\})=1$ for every $\gamma \in \Gamma$.
We denote by $\ell_p(\Gamma)$, $1\le p\le \infty$, the $L_p$-space on $(\Gamma, \Sigma, \mu)$\cite[p.xi]{LT1}.
Whereas $\ell_p(\Gamma)$ is a well-studied object (see e.g. \cite{LT1,SSU,HR98} and references therein) and the description  of surjective isometries of $\ell_p(\Gamma)$ follows yet from \cite{Yeadon,Sherman}, the case of arbitrary  symmetric spaces $E(\Gamma)$ for uncountable $\Gamma$   remained untreated.
The description of isometries of these classical Banach spaces is a simple corollary of our general result, Theorem~\ref{Theorem1} below.
We also note that
the study of isometries on real symmetric function spaces and those on complex symmetric function spaces have substantial differences (see e.g. the works  of
 Braveman and Semenov \cite{BS74,BS75},   Jaminson, Kaminska and Lin\cite{JKL}, and
 Kalton and  Randrianantoanina\cite{Kalton_R,Kalton_R93,R,R2}).
Throughout this paper, unless stated otherwise,  we only consider complex Banach spaces and surjective linear isometries.

A noncommutative version of   Banach's description on isometries between $L_p$-spaces\cite{Banach} was obtained by
 Kadison \cite{K51} in the 1950s, who showed that  a surjective isometry between two von Neumann algebras can be written as
 a Jordan $*$-isomorphism  followed by a multiplication of    a unitary operator.
After the non-commutative $L_p$-spaces were introduced by Dixmier\cite{D50} and Segal\cite{Se} in the 1950s,
the study of $L_p$-isometries was  conducted by Broise~\cite{Broise}, Russo \cite{Russo}, Arazy \cite{Arazy}, Tam \cite{Tam},  etc.
A complete description (for the semifinite case) was  obtained in 1981 by Yeadon \cite{Yeadon}, who proved that
every  isometry  $T:L_p(\cM_1,\tau_1) \stackrel{into}{\longrightarrow} L_p(\cM_2,\tau_2)$, $1\le p\ne 2<\infty$,
has the form
\begin{align}\label{formp}
T(x)=uB J(x), ~x\in \cM_1\cap L_p(\cM_1,\tau_2),
\end{align}
where $u$ is  a partial isometry  in $\cM_2$, $B$ is a positive self-adjoint operator affiliated with $\cM_2$ and $J$ is  a Jordan $*$-isomorphism from $\cM_1$ onto a weakly closed $*$-subalgebra of $\cM_2$ (see \cite{Sherman,JS,JRS, Watanabe} for the case when $\cM_1,\cM_2$ are of type $III$).

The isometries on general symmetric operator spaces on semifinite von Neumann algebras  have been widely   studied    since the notion of symmetric operator spaces was introduced in the 1970s (see e.g. \cite{O,O2,DDP,S87,Kalton_S,DDP93} and references therein).
The question posed at the beginning of the paper indeed asks whether  these  isometries $T$
have a natural description as in the cases of symmetric function spaces and noncommutative $L_p$-spaces (see \eqref{formf} and \eqref{formp}).
 One of the most important developments  in this area is due to the work of  Sourour\cite{Sourour}, who described isometries on separable symmetric  operator ideals, that is, when $\cM$ is the $*$-algebra $B(\cH)$ of all bounded linear operator on a separable Hilbert space $\cH$.
 Adopting Sourour's techniques, the second author obtained the description of isometries on separable symmetric operator spaces affiliated with hyperfinite type $II$ factors\cite{Sukochev}.
 However,  the approach used in \cite{Sourour} strongly relies  on the matrix representation of  compact operators on a separable Hilbert space $ \cH$, which is not applicable for symmetric operator spaces affiliated with general semifinite von Neumann algebras.
In the latter case, only partial results have been obtained.
For example, the general form of isometries of Lorentz spaces on a finite von Neumann algebra was obtained in \cite{CMS} (see also \cite{MS2}).
Under additional  conditions on the isometries (e.g., disjointness-preserving, order-preserving, etc.), similar descriptions  can be found in \cite{MS, JC, JC2, SV,  CMS,HSZ,FJ2,Katavolos2,MZ,R3,Abramovich}, which provide partial  answers to the question posed at the outset of this paper.

The following theorem answers Question \ref{mainq} in its  full generality..
\begin{theorem}\label{Theorem1}
Let $\cM_1$ and $\cM_2$ be  atomless von Neumann algebras (or  atomic von Neumann algebras whose atoms all  have the same trace) equipped with semifinite faithful normal traces $\tau_1$ and $\tau_2$, respectively.
Let $E(\cM_1,\tau_1)$ and $F(\cM_2,\tau_2)$ be two  symmetric operator spaces whose norms  are order continuous and are not proportional to  $\norm{\cdot}_2$.
If $T:E(\cM_1,\tau_1)\to F(\cM_2,\tau_2)$ is a surjective isometry, then there exist
two nets of elements $ A_i \in  F(\cM_2 ,\tau_2),~   i\in I $,  disjointly supported from the right and $ B_i \in F(\cM_2 ,\tau_2), ~i\in I$,    disjointly supported from the  left,   a  surjective  Jordan $*$-isomorphism $J:\cM_1\to \cM_2$  and a central projection $z\in \cM_2$ such that
$$T(x) =  \norm{\cdot}_F-  \sum _{i\in I}  J(x)A_i z + B_i J(x) ({\bf 1}-z) , ~\forall x\in E(\cM_1,\tau_1)\cap \cM_1,$$
 where   the series is taken as the limit of  all finite partial sums.
 In particular, if
$\cM$ is $\sigma$-finite, then the nets $\{A_i\}$ and $\{B_i\}$ are countable.
If the trace $\tau$ is finite, then there exist   elements $A,B\in F(\cM_2,\tau_2)$ such that
$$T(x) =    J(x)A  z + B  J(x) ({\bf 1}-z) , ~\forall x\in E(\cM_1,\tau_1)\cap \cM_1. $$
\end{theorem}
This extends numerous earlier results in this topic 
 (see e.g. \cite{Arazy85,JC,JC2, Sourour,Sukochev,CMS,Russo,Lumer, Sherman,Watanabe,Yeadon,MS,MS2,Medzhitov}),  and
Theorem \ref{Theorem1} yields the first description of surjective isometries on symmetric operator spaces associated with non-hyperfinite algebras.
On the other hand, we show that if $\cM$ has atoms whose traces are different, then
there exists a symmetric space $E(\cM,\tau)$ (whose norm is  not proportional to $\norm{\cdot}_2$) and  an   isometry  on $E(\cM,\tau)$
which is  not in the form of \eqref{formf} (see Example~\ref{exam}). 
This  demonstrates that the assumption imposed on  the von Neumann algebra is sharp.

Recall that the notion of hermitian operators on a Banach space was formulated by Lumer\cite{Lumer} in his seminal paper in the 1960s,
for the purpose of extending Hilbert space type arguments to Banach spaces.
This notion plays
  an important role in  different fields  such as
   operator theory on Banach spaces, matrix theory, optimal control theory and computer science
   (see e.g. \cite{FJ,Giles,Lumer,Berkson,FJ,FJ2,XY,Sinclair} and references therein).

The main method used in this paper  for the description of isometries
is to establish and employ the general description of hermitian operators on the symmetric operator spaces $E(\cM,\tau)$.
The following result is rather surprising as it shows that the stock of hermitian operators does not depend on the symmetric space $E(\cM,\tau)$, and it is fully determined by the algebra $\cM$.

\begin{theorem} \label{tH2}
Let $E(\cM,\tau)$ be a symmetric  space on an atomless semifinite von Neumann algebra (or an atomic von Neumann
algebra with all atoms having the same trace) $\cM$ equipped
with a semifinite faithful normal trace $\tau$.
Assume that $\norm{\cdot}_E$ is order continuous and is not proportional to $\norm{\cdot}_2$.
Then, a bounded operator $T$ on $E(\cM,\tau)$ is a hermitian operator on $E(\cM,\tau)$ if and only if  there exist   self-adjoint operators $a$ and $b$ in $\cM$ such that
$$Tx=ax+xb,~x\in E(\cM,\tau).$$
In particular, $T$ can be extended to a bounded hermitian operator on the von Neumann algebra $\cM$.
\end{theorem}
This idea to employ hermitian operators for description of isometries  lurks in the background of Lumer's  description
of isometries on Orlicz spaces\cite{FJ},
however, the study of hermitian operators on noncommutative spaces is  substantially more difficult than  that of function spaces,
and
 descriptions of hermitian operators are known only for  very few   operator spaces.
For example,
Sinclair obtained the general form of hermitian operators  on a $C^*$-algebra\cite{Sinclair} using earlier results on derivations on operator algebras;
 Sourour obtained the general form of hermtitian operator on separable operator ideal of $B(\cH)$ when $\cH$ is separable\cite{Sourour};
 and the case for symmetric operator spaces on   hyperfinite type $II$ factors  was obtain by the second author\cite{Sukochev} by adopting Sourour's approach.
 For more general von Neumann algebras, the   form of a hermitian operator on a symmetric space (even on  noncommutative
 $L_p$-spaces,  see \cite[Theorem 4]{Sourour77} and \cite[Theorem 4.2]{Sourour78} for partial results) was unknown.
Theorem \ref{tH2} yields  the complete description of  hermitian operators on a symmetric operator space  having order continuous norm by using a different approach to those  in \cite{Sourour,Sukochev}.

 The main ingredient of the proof of Theorem \ref{tH2} is the following surprising observation:
any bounded hermitian operator on $E(\cM,\tau)$ can be ``reduced'' to a bounded hermitian operator on   the so-called $\tau$-compact ideal $C_0(\cM,\tau)$ (which is a $C^*$-algebra) and therefore, it can be written as  the sum of a left-multiplication by a self-adjoint operator in $\cM$ and
a right-multiplication by
a self-adjoint operator in $\cM$\cite{Sinclair}.
 Having  such a description
at hand, we are able to describe isometries on symmetric operator spaces affiliated with   $\cM$ and infer Theorem \ref{Theorem1}.
However,
 the structure of a bounded hermitian operator on  a von Neumann algebra
 is more complicated than that of    a factor.
 This inference is, however, far from straightforward.
As pointed out in   \cite[p.825]{JC}, constructing a suitable Jordan $*$-isomorphism from an isometry (even if this isometry is positive and finiteness preserving\cite{JC}) is always `problematic'.
A simple adaptation of proofs in \cite{Sourour,Sukochev} does not  yield   Theorem \ref{Theorem1}.
Many   new techniques are required in the proof of Theorem \ref{Theorem1},
 which are of interest on their  own rights  and has potential  usage in   the future   study  of hermitian operators, Jordan $*$-isomorphisms and isometries of vector-valued spaces.

Finally, as an application of Theorem  \ref{Theorem1},
we
consider a variant of  Pe{\l}czy\'{n}ski's problem on the uniqueness of symmetric structure of operator ideals
for  symmetric structure of $E(\cM,\tau)$ affiliated with a $II_1$-factor,
 which establishes a noncommutative version of  a  result by Abramovich and Zaidenberg for the uniqueness of symmetric structure of $L_p(0,1)$ \cite[Theorem 1]{AZ} and its generalizations
due to  Zaidenberg \cite{Z77}, and    Kalton and  Randrianantoanina  \cite{Kalton_R,R}.

The authors would like to thank Professor  Aleksey Ber for helpful discussions, and thank Professors Evgueni  Semenov and Mikhail  Zaidenberg for their interest and points to   the existing literature  and related problems in this field.
We also thank  Professor Dmitriy  Zanin for pointing out a gap in our original  proof of Corollary~\ref{4.2} in  the earlier version of this paper and providing Appendix~\ref{appendix}.
We thank Thomas Scheckter for his careful reading of this paper.

\section{Preliminaries}\label{s:p}

In this section, we recall main notions of the theory of noncommutative integration, introduce some properties of generalised singular value functions and define noncommutative symmetric operator spaces.
For details on von Neumann algebra
theory, the reader is referred to e.g. \cite{Blackadar}, \cite{Dixmier}, \cite{KR}
or \cite{Tak}. General facts concerning measurable operators may
be found in \cite{Nelson}, \cite{Se}, \cite{DP2014} (see also the forthcoming book \cite{DPS}).
For convenience of the reader, some of the basic definitions are recalled.

\subsection{$\tau$-measurable operators and generalised singular values}
In what follows,  $\cH$ is a (not necessarily separable) Hilbert space and $\left(B(\cH),\norm{\cdot}_\infty  \right)$ is the
$*$-algebra of all bounded linear operators on $\cH$, and
$\mathbf{1}$ is the identity operator on $\cH$.
Let $\mathcal{M}$ be
a von Neumann algebra on $\cH$.
Let   $\cP\left(\mathcal{M}\right)$  be the set of all
projections of $\mathcal{M}$.
We denote  by $\cM_p$ the reduced von Neumann algebra $p\cM p$ generated by a projection $p\in \cP(\cM)$.

\label{s:p1}
A linear operator $x:\mathfrak{D}\left( x\right) \rightarrow \cH $,
where the domain $\mathfrak{D}\left( x\right) $ of $x$ is a linear
subspace of $\cH$, is said to be {\it affiliated} with $\mathcal{M}$
if $yx \subseteq xy$ for all $y\in \mathcal{M}^{\prime }$, where $\mathcal{M}^{\prime }$ is the commutant of $\mathcal{M}$.
A linear
operator $x:\mathfrak{D}\left( x\right) \rightarrow \cH $ is termed
{\it measurable} with respect to $\mathcal{M}$ if $x$ is closed,
densely defined, affiliated with $\mathcal{M}$ and there exists a
sequence $\left\{ p_n\right\}_{n=1}^{\infty}$ in the set  $\cP\left(\mathcal{M}\right)$ of all
projections of $\mathcal{M}$ such
that $p_n\uparrow \mathbf{1}$, $p_n(\cH)\subseteq\mathfrak{D}\left(x \right) $
and $\mathbf{1}-p_n$ is a finite projection (with respect to $\mathcal{M}$)
for all $n$.
It should be noted that the condition $p _{n}\left(
\cH\right) \subseteq \mathfrak{D}\left( x\right) $ implies that
$xp _{n}\in \mathcal{M}$. The collection of all measurable
operators with respect to $\mathcal{M}$ is denoted by $S\left(
\mathcal{M} \right) $, which is a unital $\ast $-algebra
with respect to strong sums and products (denoted simply by $x+y$ and $xy$ for all $x,y\in S\left( \mathcal{M}\right) $).

Let $x$ be a self-adjoint operator affiliated with $\mathcal{M}$.
We denote its spectral measure by $\{e^x\}$.
It is well known that if
$x$ is a closed operator affiliated with $\mathcal{M}$ with the
polar decomposition $x = u|x|$, then $u\in\mathcal{M}$ and $e\in
\mathcal{M}$ for all projections $e\in \{e^{|x|}\}$.
Moreover,
$x \in S(\mathcal{M})$ if and only if $x $ is closed,
 densely
defined, affiliated with $\mathcal{M}$ and $e^{|x|}(\lambda,
\infty)$ is a finite projection for some $\lambda> 0$.
 It follows
immediately that in the case when $\mathcal{M}$ is a von Neumann
algebra of type $III$ or a type $I$ factor, we have
$S(\mathcal{M})= \mathcal{M}$.
For type $II$ von Neumann algebras,
this is no longer true.
From now on, let $\mathcal{M}$ be a
semifinite von Neumann algebra equipped with a faithful normal
semifinite trace $\tau$.

For any closed and densely defined linear operator $x :\mathfrak{D}\left( x \right) \rightarrow \cH $,
the \emph{null projection} $n(x)=n(|x|)$ is the projection onto its kernel $\mbox{Ker} (x)$.
The \emph{left support} $l(x )$ is the projection onto the closure of its range $\mbox{Ran}(x)$ and the \emph{right support} $r(x)$ of $x$ is defined by $r(x) ={\bf{1}} - n(x)$.

An operator $x\in S\left( \mathcal{M}\right) $ is called $\tau$-measurable if there exists a sequence
$\left\{p_n\right\}_{n=1}^{\infty}$ in $\cP \left(\mathcal{M}\right)$ such that
$p_n\uparrow \mathbf{1}$, $p_n\left( \cH \right)\subseteq \mathfrak{D}\left(x \right)$ and
$\tau(\mathbf{1}-p_n)<\infty $ for all $n$.
The collection of all $\tau $-measurable
operators is a unital $\ast $-subalgebra of $S\left(
\mathcal{M}\right) $,  denoted by $S\left( \mathcal{M}, \tau\right)
$.
It is well known that a linear operator $x$ belongs to $S\left(
\mathcal{M}, \tau\right) $ if and only if $x\in S(\mathcal{M})$
and there exists $\lambda>0$ such that $\tau(e^{|x|}(\lambda,
\infty))<\infty$.
Alternatively, an unbounded operator $x$
affiliated with $\mathcal{M}$ is  $\tau$-measurable (see
\cite{FK}) if and only if
$$\tau\left(e^{|x|}(n,\infty)\right)\rightarrow 0,\quad n\to\infty.$$

\begin{definition}
Let a semifinite von Neumann  algebra $\mathcal M$ be equipped
with a faithful normal semi-finite trace $\tau$ and let $x\in
S(\mathcal{M},\tau)$. The generalised singular value function $\mu(x):t\rightarrow \mu(t;x)$, $t>0$,  of
the operator $x$ is defined by setting
$$
\mu(t;x)
=
\inf \left\{
\left\|xp\right\|_\infty :\ p\in \cP(\cM), \ \tau(\mathbf{1}-p)\leq t
\right\}.
$$
\end{definition}
An equivalent definition in terms of the
distribution function of the operator $x$ is the following. For every self-adjoint
operator $x\in S(\mathcal{M},\tau) $,  setting
$$d_x(t)=\tau(e^{x}(t,\infty)),\quad t>0,$$
we have (see e.g. \cite{FK} and \cite{LSZ})
$$
\mu(t; x)=\inf\{s\geq0:\ d_{|x|}(s)\leq t\}.
$$
Note that $d_x(\cdot)$ is a right-continuous function (see e.g.  \cite{FK} and \cite{DPS}).

Consider the algebra $\mathcal{M}=L^\infty(0,\infty)$ of all
Lebesgue measurable essentially bounded functions on $(0,\infty)$.
Algebra $\mathcal{M}$ can be viewed as an abelian von Neumann
algebra acting via multiplication on the Hilbert space
$\mathcal{H}=L^2(0,\infty)$, with the trace given by integration
with respect to Lebesgue measure $m.$
It is easy to see that the
algebra of all $\tau$-measurable operators
affiliated with $\mathcal{M}$ can be identified with
the subalgebra $S(0,\infty)$ of the algebra of Lebesgue measurable functions which consists of all functions $f$ such that
$m(\{|f|>s\})$ is finite for some $s>0$. It should also be pointed out that the
generalised singular value function $\mu(f)$ is precisely the
decreasing rearrangement $\mu(f)$ of the function $|f|$ (see e.g. \cite{KPS,Bennett_S
}) defined by
$$\mu(t;f)=\inf\{s\geq0:\ m(\{|f|\geq s\})\leq t\}.$$

For convenience of the reader,  we also recall the definition of the \emph{measure topology} $t_\tau$ on the algebra $S(\cM,\tau)$. For every $\varepsilon,\delta>0,$ we define the set
$$V(\varepsilon,\delta)=\{x\in S(\mathcal{M},\tau):\ \exists p \in \cP\left(\mathcal{M}\right)\mbox{ such that }
\left\|x(\mathbf{1}-p)\right\|_\infty \leq\varepsilon,\ \tau(p)\leq\delta\}.$$ The topology
generated by the sets $V(\varepsilon,\delta)$,
$\varepsilon,\delta>0,$ is called the measure topology $t_\tau$ on $S(\cM,\tau)$ \cite{DPS, FK, Nelson}.
It is well known that the algebra $S(\cM,\tau)$ equipped with the measure topology is a complete metrizable topological algebra \cite{Nelson}.
We note that a sequence $\{x_n\}_{n=1}^\infty\subset S(\cM,\tau)$ converges to zero with respect to measure topology $t_\tau$ if and only if $\tau\big( e  ^{|x_n|}(\varepsilon,\infty)\big)\to 0$ as $n\to \infty$ for all $\varepsilon>0$ \cite{DPS}.

The space  $S_0(\cM,\tau)$ of $\tau$-compact operators is the space associated to the algebra of functions from $S(0,\infty)$ vanishing at infinity, that is,
$$S_0(\cM,\tau) = \{x\in S(\cM,\tau) :  \ \mu(\infty; x) =0\}.$$
The two-sided ideal $\cF(\tau)$ in $\cM$ consisting of all elements of $\tau$-finite range is defined by
$$\cF(\tau)=\{x\in \cM ~:~ \tau(r(x)) <\infty\} = \{x \in \cM ~:~ \tau(s(x)) <\infty\}.$$
Note that  $S_0(\cM,\tau)$ is the closure of $\cF(\tau)$ with respect to the measure topology \cite{DP2014}.

A further important vector space topology on $S(\cM,\tau)$ is the \emph{local measure topology} \cite{DP2014,DPS}.
A neighbourhood base for this topology is given by the sets $V(\varepsilon, \delta; p )$, $\varepsilon, \delta>0$, $p\in \cP(\cM)\cap \cF(\tau)$, where
$$V(\varepsilon,\delta;  p ) = \{x\in S(\cM,\tau): pxp \in V(\varepsilon,\delta)\}. $$
It is clear that the localized  measure topology is weaker than the measure topology~\cite{DP2014,DPS}.
If $\{x_\alpha\}\subset S(\cM,\tau)$ is a net and if $x_\alpha \rightarrow_\alpha x \in S(\cM,\tau)$ in local measure topology, then $x_\alpha y\rightarrow xy $ and $yx _\alpha \rightarrow yx $ in the local measure topology for all $y \in S(\cM,\tau)$ \cite{DP2014,DPS}.
If $0\le a_i $ is an increasing net in $S(\cM,\tau)$ and if $a\in S(\tau)$ is such that $a=\sup a_i$, then we write $0\le a_i \uparrow a$\cite[p.212]{DP2014}.
If $\{x_i\}$ is an increasing net in $S(\cM,\tau)_+$ and $x \in S(\cM,\tau)_+$ such that $x_i\to x$ in the local measure topology, then $x_i\uparrow x$ (see e.g. \cite[Chapter II, Proposition 7.6 (iii)]{DPS}).

\subsection{Symmetric  spaces of $\tau$-measurable operators}

Let $E(0,\infty)$  be a Banach space of real-valued Lebesgue measurable
functions on  $(0,\infty)$ (with identification
$m$-a.e.), equipped with a  norm $\left\|\cdot\right\|_E$.
The space $E(0,\infty)$ is said to be {\it
absolutely solid} if $x\in E(0,\infty)$ and $|y|\leq |x|$, $y\in S(0,\infty)$
implies that $y\in E(0,\infty)$ and $\|y\|_E\leq\|x\|_E.$
An absolutely solid space $E(0,\infty)\subseteq S(0,\infty)$ is said to be {\it
symmetric} if for every $x\in E(0,\infty)$ and every $y\in S(0,\infty)$,
 the assumption
$\mu(y)=\mu(x)$ implies that $y\in E(0,\infty)$ and $\left\|y\right\|_E=\left\|x\right\|_E$
\cite{KPS}.
Without of loss generality,
throughout this paper, we always assume that $\norm{\chi_{(0,1)}}_{E(0,\infty)}=1$.

We now come to the definition of the main object of this paper.
\begin{definition}\label{opspace}
Let $\cM $ be a semifinite von Neumann  algebra  equipped
with a faithful normal semi-finite trace $\tau$.
Let $\mathcal{E}$ be a linear subset in $S({\mathcal{M}, \tau})$
equipped with a complete norm $\left\|\cdot \right \|_{\mathcal{E}}$.
We say that
$\mathcal{E}$ is a \textit{symmetric    space}  if
for $x \in
\mathcal{E}$, $y\in S({\mathcal{M}, \tau})$ and  $\mu(y)\leq \mu(x)$ imply that $y\in \mathcal{E}$ and
$\left\|y\right\|_\mathcal{E}\leq \left\|x\right\|_\mathcal{E}$.
\end{definition}
Let $E(\cM,\tau)$ be a symmetric   space.
It is well-known that any symmetrically normed space $E(\cM,\tau)$ is a normed $\cM$-bimodule (see e.g.  \cite{DP2014} and \cite{DPS}). 
That is, for any symmetric operator space $E(\cM,\tau)$, we have
$\left\|axb\right\|_E \le \left\|a\right\|_\infty \left\|b\right\|_\infty \left\|x\right\|_E, ~a, b \in \cM,~ x\in E(\cM,\tau)$.
It is known that whenever $E(\cM,\tau)$ has   order continuous norm $\norm{\cdot}_E$, i.e.,  $\norm{x_\alpha}_E\downarrow 0$ whenever $0\le x_\alpha \downarrow 0\subset E(\cM,\tau)$, we  have $E(\cM,\tau)\subset S_0(\cM,\tau)$\cite{DPS,HSZ,DP2014}.

 The so-called K\"{o}the dual is identified with an important part of the dual space.
If $E (\cM,\tau) \subset S(\cM,\tau)$ is a symmetric space, then the K\"{o}the dual $E(\cM,\tau)^\times $ of $E$ is defined by setting
$$ E(\cM,\tau)^\times =\left\{   x\in S(\cM,\tau) : \sup_{\|y\|_E\le 1, y\in E(\cM,\tau)}\tau (|xy|)   <\infty    \right\}.$$
The K\"{o}the dual $E(\cM,\tau)^\times  $ can be identified as   a subspace of the Banach dual $E(\cM,\tau)$ via the trace duality \cite[p.228]{DP2014}.
Recall that $x\in L_1(\cM,\tau)+\cM:=
 \{
a\in S(\cM,\tau):\mu(a)\in L_1(0,\infty)+L_\infty(0,\infty)
 \}$
can be equipped with
a norm $\norm{x}_{L_1+L_\infty}= \int_0^1 \mu(s;x)ds $
and $x\in L_1(\cM,\tau) \cap \cM:=
 \{a\in S(\cM,\tau):\mu(a)\in L_1(0,\infty)\cap L_\infty(0,\infty) \}$ can be equipped with a norm $\norm{x}_{L_1\cap L_\infty}:=\max \{\norm{x}_1,\norm{x}_\infty\}$.
In particular, $\left(L_1(\cM,\tau)\cap\cM\right)^\times= L_1(\cM,\tau)+\cM$ and $L_1(\cM,\tau)\cap\cM= \left( L_1(\cM,\tau)+\cM\right)^\times$\cite[Example 4]{DP2014}.

The \emph{carrier projection} $c_E\in \cM$ of   $E(\cM,\tau)$ is defined by setting
$$c_E := \bigvee \left\{p:p\in \cP(E)\right\}.$$
 It is clear that $c_E$ is in the center $Z(\cM)$ of $\cM$ \cite{DP2014}.
It is often assumed that the carrier projection $c_E$ is equal to ${\bf 1}$.
Indeed, for any  symmetric function  space  $E(0,\infty)$, the carrier projection of the corresponding operator space $E(\cM,\tau)$ is always ${\bf 1}$ (see e.g.
\cite{DP2014,Kalton_S}).
On the other hand, if $\cM$ is atomless or is atomic and all atoms have equal trace, then any non-zero symmetric space $E(\cM,\tau)$ is necessarily $\bf 1$\cite{DP2014,DPS}.
In this case, whenever  $E(\cM,\tau)$ has order continuous norm, then $E(\cM,\tau)^\times $ is isometrically isomorphic to $E(\cM,\tau)^*$ (see e.g. \cite{DDP93}, \cite[Proposition 6.4]{DP2012} or \cite[Proposition 47(v)]{DP2014}).



There exists a strong connection between symmetric function spaces and
operator spaces exposed in \cite{Kalton_S} (see also \cite{LSZ,DP2014}).
The operator space $E(\cM,\tau)$ defined by
\begin{equation*}
E(\mathcal{M},\tau):=\{x \in S(\mathcal{M},\tau):\ \mu(x )\in E(0,\infty)\},
\ \left\|x \right\|_{E(\mathcal{M},\tau)}:=\left\|\mu(x )\right\|_E
\end{equation*}
 is a complete symmetric  space  whenever $(E(0,\infty),\left\|\cdot\right\|_E)$ is    a complete  symmetric  function space on $(0,\infty)$  \cite{Kalton_S}.
In particular, for any symmetric function space $E(0,\infty)$, $F(\tau)\subset E(\cM,\tau)$\cite[Lemma 18]{DP2014}.
In the special case when $E(0,\infty)=L_p(0,\infty)$, $1\le  p\le \infty$, $E(\cM,\tau)$ is the noncommutative $L_p$-spaces affiliated with $\cM$ and we denote the norm by $\norm{\cdot}_p$.
We note that if $E(0,\infty )$ is separable (i.e. has order continuous norm), then $E^\times (\cM,\tau) $ is isometrically isomorphic to $E(\cM,\tau)^*$\cite[p.246]{DP2014}.
Recall that every separable symmetric sequence/function  space $E$ is fully symmetric, that is,
 if $x\in E$ and $y\in \ell_\infty$ (resp. $y\in S(0,\infty)$) with
 $$ \int_0^t \mu(t;y)dt \le \int_0^t\mu(t;x)dt, ~t\ge 0 $$
 (denoted by $y\prec \prec x$),
 then $y\in E$ with $\norm{y}_E\le \norm{x}_E$ (see e.g. \cite[Chapter II,Theorem 4.10]{KPS} or \cite[Chapter IV, Theorem 5.7]{DPS}).


\section{Hermitian operators}

Let $X$ be a Banach space.
Recall that a \emph{semi-inner product} (abbreviated \emph{s.i.p.}) on $X$ is a mapping $\langle \cdot,\cdot\rangle$ of $X\times X$ into the field   of complex numbers such that:
\begin{enumerate}
  \item $\langle x+y,z\rangle= \langle x,z\rangle+\langle y,z\rangle$ for $x,y,z\in X$;
  \item $\langle \lambda x,y \rangle=\lambda \langle x,y \rangle$ for $x,y\in X$ and $\lambda\in \mathbb{C}$;
  \item $\langle x,x\rangle >0$ for $0\ne x\in X$;
  \item $|\langle  x,y\rangle|^2 \le \langle x,x\rangle\langle y,y\rangle$ for any $x,y\in X$.
\end{enumerate}
When a s.i.p is defined on $X$, we call $X$ a \emph{semi-inner-product space} (abbreviated \emph{s.i.p.s.}).
If $X$ is a s.i.p.s., then $\langle x,x \rangle^{\frac12}$
 is a norm on $X$.
 On the other hand, every Banach space can be made into a s.i.p.s. (in general, in infinitely many ways) so that the s.i.p. is consistent with the norm, i.e., $\langle x,x\rangle^\frac 12 = \norm{x}$ for any $x\in X$\cite{FJ}.
 By virtue of the Hahn--Banach theorem, this can be accomplished by choosing  one bounded linear functional $f_x$ for each $x\in X$  such that $\norm{f_x} =\norm{x}$
and $f_x(x)=\norm{x}^2$ ($f_x$ is called a \emph{support functional} of $x$), and then setting $\langle x,y\rangle = f_y (x)$ for arbitrary $x,y\in X$\cite{Berkson, Lumer,FJ}.
Given a linear transformation $T$ mapping a s.i.p.s. into itself, we denote by $W(T)$ the \emph{numerical range} of $T$, that is,  $\{\langle Tx,x\rangle| \langle x,x\rangle=1, x\in X\}$.
Let $T$ be an operator on a Banach space $(X,\norm{\cdot})$. Although in
principle there may be many different s.i.p. consistent
 with $\norm{\cdot}$,
 nonetheless if the numerical range of $T$ relative to one such
s.i.p. is real, then the numerical range relative to any such s.i.p. is real  (see e.g. \cite[p.107]{FJ}, \cite[Section 6]{Lumer} and \cite[p.377]{Berkson}).
If this is the case, $T$ is said to be a \emph{hermitian} operator on $X$.

From now on, unless stated otherwise,  we  always assume that $\cM$ is an   atomless semifinite  von Neumann algebra or an atomic semifinite von Neumann algebra with all atoms having  the same traces (without loss of generality, we assume that $\tau(e)=1$ for any atom $e\in \cM$), and we assume that  $\tau$ is a semifinite faithful normal trace on $\cM$.

In particular,
when $\cM$ is atomless (resp. atomic), the set
$$E(0,\tau({\bf 1})):=\{f\in S(0,\tau({\bf 1})) :\mu(f)=\mu(x) \mbox{ for some } x\in E(\cM,\tau)\}$$
(resp.
$$\ell_E:=\{f\in  \ell_\infty  :\mu(f)=\mu(x) \mbox{ for some } x\in E(\cM,\tau)\})$$
is a symmetric function (resp. sequence) space\cite[Theorem 2.5.3]{LSZ}.
There exists a bijective correspondence between symmetric operator spaces and symmetric function/sequence spaces.
Therefore, if $\norm{\cdot}_E$ on $E(\cM,\tau)$ is not proportional to $\norm{\cdot}_2$ on $L_2(\cM,\tau)$, then $\norm{\cdot}_E$
is not proportional to   $\norm{\cdot}_2$ on $L_2(\cA,\tau)$ for any maximal abelian von Neumann subalgebra $\cA$ of $\cM$.


Sourour\cite[Lemma 1]{Sourour} obtained Lemma \ref{lemma:orthogonal} below in the setting of $B(\cH)$ by using  a result due to  Schneider and  Turner
(see e.g. \cite[Lemma 3.1]{ST} and \cite[Lemma 9.2.7]{FJ}).
Arazy gave  a self-contained alternative proof in the setting of complex sequence spaces\cite{Arazy85}.
In the proof of the following lemma,
 we adopt   Arazy's proof.
Due to the  technical differences between the atomless case and atomic case, we provide a full proof below.

 Before proceeding to the proof of Lemma \ref{lemma:orthogonal}, we need the following well-known proposition.
For the sake of completeness, we provide a short proof below.
\begin{proposition}\label{prop:tppR}
Let $p\in \cF(\tau)$ be a  projection and let $E(\cM,\tau)$ be an arbitrary  symmetric operator space having order continuous norm.
   Then,  $ \frac{\norm{p}_E  ^2 }{\tau(p)} p \in E(\cM,\tau)^*$ is a support functional of $p\in E(\cM,\tau)$,
    i.e.,
 $ \tau\left(p\cdot  \frac{\norm{p} _E^2 }{\tau(p)} p  \right) =\norm{p}_E^2  =\norm{p}_E \norm{ \frac{\norm{p} _E^2 }{\tau(p)} p }_{E^*}  $.
In particular, for any bounded hermitian operator $T$ on $E(\cM,\tau)$, we have
\begin{align}\label{prop:tppR}
\tau(T(p) p)\in \mathbb{R}.
\end{align}
\end{proposition}
\begin{proof}
We only consider the case when
$\cM$ is atomless. The atomic case follows from the same argument (see also \cite{Arazy85} or \cite[Theorem 5.2.13]{FJ}).



Note that
$$  \norm{p }_{E^*}= \norm{p }_{E^\times } = \sup \left\{  \int_{0}^{\tau(p)} \mu(s;z )ds :z\in E(\cM,\tau), ~\norm{z}_E =  1 \right\} . $$
Since $\frac{\int_0^{\tau(p)} \mu(s;z )ds  }{{\tau(p)} } \mu( p )  = \frac{\int_0^{\tau(p)} \mu(s;z )ds }{{\tau(p)} } \chi_{(0,\tau(p))}   \prec \prec \mu(z)   $, $z\in E(\cM,\tau)$, $\norm{z}_E =  1$
  (see e.g. \cite[Section 3.6]{LSZ}), we obtain that $\frac{\int_0^{\tau(p)} \mu(s;z )ds  }{{\tau(p)} }  \norm{p}_E \le \norm{z}_E =1 $, and therefore,
$$\norm{p  }_{E^*}  \le \frac{\tau(p)}{\norm{p }_E }.$$
On the other hand,    we have\cite[Remark 3]{DP2014}
$$\tau(p  )\le  \norm{p}_{E^*} \norm{p}_{E}.$$
Hence, $\tau(p  )= \norm{p}_{E^*} \norm{p}_{E}$, i.e. $\tau\left(p\cdot  \frac{\norm{p} _E^2 }{\tau(p)} p  \right) = \norm{p}_E^2 =    \norm{p}_E \norm{ \frac{\norm{p} _E^2 }{\tau(p)} p }_{E^*}  $.
\end{proof}

\begin{cor}\label{prop:tUUR}
Let $u\in \cF(\tau)$ be a  partial isometry and let $E(\cM,\tau)$ be an arbitrary  symmetric operator space having order continuous norm.
   Then,  $ \frac{ \norm
{ u^*u } _E ^2 }{\tau(u^* u)} u^*   \in E(\cM,\tau)^*$ is a support functional of $u\in E(\cM,\tau)$.
In particular, for any bounded  hermitian operator $T$ on $E(\cM,\tau)$, we have
\begin{align*}
\tau(T(u) u^*)\in \mathbb{R}.
\end{align*}
\end{cor}
\begin{proof}
Since $u\in \cF(\tau)$, it follows  that $r(u)$ and $l(u)$ are $\tau$-finite projections.
Hence, $r(u)\vee l(u)$ are also $\tau$-finite.
Therefore, there exists a unitary element $v$ in $  \cM _{r(u)\vee l(u)} $ such that $v^*  l(u)v =r(u)$\cite[Chapter XIV, Lemma 2.1]{T3}.
Define $v':= u+ v\cdot (r(u)\vee l(u) -r(u))$.
Note that
$$(r(u)\vee l(u) -r(u)) v^*u= (r(u)\vee l(u) -r(u)) v^*l(u)  u =(r(u)\vee l(u) -r(u)) r(u)v^*u =0.$$
We have
\begin{align*}
(v')^*v' &= \left ( u+ v\cdot (r(u)\vee l(u) -r(u)) \right)^* \left(u+ v\cdot (r(u)\vee l(u) -r(u))\right) \\
&=r (u )  +(r(u)\vee l(u) -r(u)) = r(u)\vee l(u) .
\end{align*}
Hence,   $v'$ is a unitary element in  $  \cM_{ r (u)\vee l(u) } $ and therefore,
 $$v'':= v'+ ({\bf 1}- r(u)\vee l(u)  )$$ is a unitary element in $\cM$.
  It follows that $( v'' )^*   T (  v'' \cdot ) $  on $E(\cM,\tau)$ is also a bounded hermitian operator\cite[p.22]{FJ}.
Noting that $v'' r(u) =u$, we obtain
  $$  \tau(T (u)  u^* ) =  \tau\left (
  ( v'' )^*   T  \left ( v''    r ( u )\right) r(u)
   \right ) \stackrel{\eqref{prop:tppR}}{\in} \mathbb{R}. $$
\end{proof}

\begin{lemma}
\label{lemma:orthogonal}
Let $E(\cM,\tau)$ be a   symmetric space affiliated with $\cM$, whose norm is order continuous and is not proportional to $\norm{\cdot}_2$.
Let $x_1,x_2\in \cF(\tau )$   be two partial isometries such that $l(x_1) \perp l(x_2)$ and $r(x_1)\perp r(x_2)$.
Then, for any bounded hermitian operator $T:E(\cM,\tau)\to E(\cM,\tau )$, we have
$$\tau (T(x_1)x _2^* ) = 0.$$
Consequently, if $x_1\in E(\cM,\tau)$ and $x_2\in L_1(\cM,\tau)\cap \cM$ with  $l(x_1) \perp l(x_2)$ and $r(x_1)\perp r(x_2)$, then $$\tau (T(x_1)x _2^* ) = 0.$$
\end{lemma}
\begin{proof}
We only prove the case for atomless von Neumann algebra. The atomic case follows by a  similar argument.

We first consider the case when $x_1$ and $x_2$ are two projections such that $x_1x_2=0$.

Since $\norm{\cdot}_E$ is not proportional to $\norm{\cdot}_2$, it follows that there exists a set of pairwise orthogonal  projections $\{e_i\}_{1\le i\le n}$ having the same trace such that $\norm{\cdot}_E$ on $E(\cA)$ is not proportional to $\norm{\cdot}_2$ on  $L_2(\cA)$,
where $\cA$ is
 the abelian weakly closed $*$-algebra generated by $\{e_i\}_{1\le i\le n}$.
Let
$$t_{1,2}:= \tau(T(e_1)e_2 )~ {\rm and } ~  t_{2,1}:= \tau(T(e_2) e_1   ) .$$



 By Proposition \ref{prop:tppR},
we obtain that $\tau(T(e_1)e _1 ),$ $ \tau(T(e_2)e_2 )\in \mathbb{R}$.
We claim that
$$t_{i,j}=\overline{t_{j,i}}  $$
when
$i,j=1,2, \cdots, n, $ and  $i\ne j$.
Define $x_\theta:= e_1 + e^{i\theta }  e_2 $, $0 \le \theta \le 2\pi  $.
In particular,
$\tau(x_\theta x_\theta^* )= \tau(e_1+e _2 )\in \mathbb{R}$.
By Corollary \ref{prop:tUUR},
we obtain that $$\tau(T(x_\theta) x_\theta^*)\in \mathbb{R},$$
i.e., $e^{i\theta } t_{i,j} +e^{-i \theta }t_{j,i}\in \mathbb{R}$  for all $\theta $.
Hence,
\begin{align}\label{conju}
t_{i,j}=\overline{t_{j,i}}  .
\end{align}

By
\cite[Lemma 4]{Arazy85} (see also \cite{AC,FJ}),  there exists $1<n<\infty$,  $x=\sum_{k=1}^n x(k) e_k$, $y=\sum_{k=1}^n y (k) e_k   $
so that
\begin{enumerate}
  \item $x(k)\ge 0$, $y(k)\ge 0$ for all $k$;
  \item $\norm{x}_E =\norm{y}_{E^*} = \tau(xy^*)= 1 $;
  \item $x$ and $y$ are linearly independent.
\end{enumerate}
Let $\theta=(\theta_1,\cdots, \theta_n)$, $0\le \theta_k \le 2\pi $, we let
$$x_\theta =\sum_{k=1}^n e^{i\theta _k } x(k)e_k  $$
and
$$y_\theta =\sum_{k=1}^n e^{i\theta _k } y(k)e_k . $$
By the assumption on  $x$ and $y$, we obtain that
$$\norm{x_\theta}_E=\norm{y_\theta}_{E^*} =\tau(y_\theta^* x_\theta )\stackrel{(2)}{=}1. $$
Therefore, by the definition of a hermitian operator, for any $\theta$, we have
\begin{align*}
0&~= ~{\rm Im} \tau(y_\theta^* T(x_\theta)) \\
&~=~ {\rm Im} \sum_{k,l=1}^n e^{i(\theta_k -\theta_l)} x(k )y(l) t_{l,k}\\
 & \stackrel{\eqref{conju}}{=}
\frac{1}{2i} \sum_{k\ne  l} x(k )y(l) (e^{i(\theta_k -\theta_l)}t_{l,k}  -e^{i(\theta_l -\theta_k)}t_{k,l} ) \\
&~=~
\frac{1}{2i} \sum_{k\ne  l} x(k )y(l)  e^{i(\theta_k -\theta_l)}t_{l,k}  - \frac{1}{2i} \sum_{k\ne  l} x(k )y(l)  e^{i(\theta_l -\theta_k)}t_{k,l}    \\
&~=~
\frac{1}{2i} \sum_{k\ne  l} x(k )y(l)  e^{i(\theta_k -\theta_l)}t_{l,k}  - \frac{1}{2i} \sum_{k\ne  l} x(l )y(k )  e^{i(\theta_k -\theta_l)}t_{l,k}    \\
&~=~
\frac{1}{2i} \sum_{k\ne  l} e^{i(\theta_k -\theta_l)} t_{l,k} (x(k )y(l)  -x(l)y(k ) )
\end{align*}
This implies that $t_{l,k} (x(k )y(l)  -x(l)y(k ) )  =0$ for all $k$ and $l $.

Since $x,y$ are linearly  independent, it follows that there exists $k\ne l$ such that $x(k )y(l)  -x(l)y(k )\ne 0$ and thus $t_{l,k}=0$.
Replacing in this argument $x$ and $y$ by $x_\pi= \sum x_{k}e_{\pi(k)}$ and $y_{\pi}= \sum y_{k}e_{\pi(k)}$, respectively, where $\pi$ is an arbitrary permutation of $\{1,\cdots, n\}$, we deduce that $t_{k,l}=0$ for every $k\ne l$.
Therefore, we obtain that
$$\tau(T(p)q)=0$$
for
any   projections $p,q$ with $pq=0$ and $\tau(p)=\tau(q)=\tau(e_1)$.
 Clearly, if the space $(E(\cA),\norm{\cdot}_{E(\cM,\tau)})$ on the algebra generated by  $e_k$, $1\le k\le n$, is not proportional to $\norm{\cdot}_2$, then the norm $\norm{\cdot}_E$ on
 $E(\cA')$ on the algebra generated by
 $e_k'$, $1\le k\le 2n$, is not proportional to $\norm{\cdot}_2$ as well,
 where $\tau(e_k')=\frac12\tau(e_k)$.
This implies that for any $k\in \mathbb{N}$, if $p,q $ are $\tau$-finite projections such that $pq=0$ and  $\tau(p)=\tau(q)= \frac{1}{2^k }\tau(e_1)$, then
\begin{align}\label{Tpq0}
\tau(T(p)q)=0.
\end{align}

The case when  $\tau(p )\ne \tau( q )$ follows by standard approximation argument.
Indeed,
let    $p,q  $ be  $\tau$-finite projections with $pq=0$.
For any $\varepsilon>0$,  there exist two sets of   $\tau$-finite projections
 $\{p_i\}_{1\le i\le n}$ and $\{q_i\}_{1\le i\le m}$  such that $$\tau\left(p-\sum_{1\le i\le n}  p_i \right), \tau\left(q-\sum_{1\le i\le m}  q_i \right)\le \varepsilon $$ and
 $$ \tau(p_i)=\tau(q_j)=\frac{1}{2^k }\tau(e_1), ~ 1\le i\le n, ~ 1\le j\le m  $$ for some $k\in \mathbb{N}$.
Note that
 \begin{align*}
~&  \tau\left( T(p) q\right)  \\
  =~& \tau\left( T(p  ) \left(q -\sum_{1\le i\le m} q_i \right)\right) + \tau\left( T\left(p -\sum_{1\le i\le n} p_i\right) \sum_{1\le i\le m} q_i   \right) +  \tau\left( T\left(\sum_{1\le i\le n} p_i \right) \sum_{1\le i\le m} q_i\right)\\
   \stackrel{\eqref{Tpq0}}{=} & \tau\left( T(p  ) \left(q -\sum_{1\le i\le n} p _i \right)\right) + \tau\left( T\left(p  -\sum_{1\le i\le n} p_i\right) \sum_{1\le i\le m } q _i   \right).
   \end{align*}
   Since
$$
\left|\tau\left( T(p ) \left(q -\sum_{1\le i\le m} q_i \right)\right)\right|
\le
\tau\left( \left|T(p  ) \left(q -\sum_{1\le i\le m} q_i \right )\right| \right)= \norm{ T(p  ) \left(q -\sum_{1\le i\le m} q_i \right) }_1 \to 0 $$
as $\varepsilon\to 0$ (see e.g. \cite[Lemma 3.10]{DPS2016} and \cite{DP2012})
and
\begin{align*}
 \left|\tau\left( T\left(p  -\sum_{1\le i\le n} p_i\right) \sum_{1\le i\le m} q _i  \right )\right|
 &\le \norm{T}\norm{p  -\sum_{1\le i\le n} p_i}_E\norm{\sum_{1\le i \le m } q _i }_{E^\times }\\
 &\le \norm{T}\norm{p  -\sum_{1\le i\le n} p_i}_E\norm{\sum_{1\le i\le m } q  }_{E^\times }\to 0
 \end{align*}
 as $\varepsilon\to 0$ because the norm $\norm{\cdot}_E$
is order continuous, it follows that
\begin{align*}
\left| \tau\left( T\left(p\right) q\right)\right|=0  .
\end{align*}

The general case when $x_1$ and $x_2$ are partial isometries in $\cF(\tau )$ such that $l(x_1) \perp l(x_2)$ and $r(x_1)\perp r(x_2)$ is reduced to the just considered case  via the same argument as in
  Corollary \ref{prop:tUUR}. Therefore,  we obtain that 
 $$  \tau\left(T  ( x_1)  x_2 ^*  ) \right )=0  ,$$
  which completes the proof of the first assertion.

Now, we prove  the second assertion.
Since $E(\cM,\tau)\subset S_0(\cM,\tau)$,  there exist two sequences $\{y_n\}$ and $\{y_n'\}$ in $\cF(\tau)$ such that $0\le y_n \uparrow |x_1|$ and $0\le y_n' \uparrow |x_2|$ and $y_n$ (resp. $y_n'$) are generated by spectral projections of $|x_1|$ (resp. $|x_2|$).
Let $x_1=u_1 |x_1|$ and $ x_2 =u_2 |x_2|$ be the polar decompositions.
By the first assertion of the lemma, we obtain that
$  \tau(T  ( u_1y_n)  (u_2y_m')^*  )=0  $
for every $n$ and $m$.
Since $ u_1y_n \to_n  x_1$ in $\norm{\cdot}_E$, it follows that
$T  ( u_1y_n) \to T(x_1)$  in $\norm{\cdot}_E$, and therefore $T  ( u_1y_n) \to T(x_1)$ weakly. Hence,
$  \tau(T  (x_1)  (u_2y_m')^*  )=0  $ for each $m$.

Consider the special case when
 $x_2\in \cF(\tau)$.
 In this case, we may assume, in addition,  that $\norm{y_m'-x_2}_\infty \to 0$.
 We obtain that
\begin{align}\label{0F}
|\tau(T(x_1)   |x_2|  u_2^* )| =|\tau(T(x_1)r(x_2) (y_m '- |x_2|) u_2^* )|\le   \norm{T(x_1)r(x_2)    }_1 \norm{y_m '- |x_2|}_\infty \to 0.
\end{align}

For the general case, let $p:=E^{|x_2|}( \delta,\infty )$, $\delta>0$,  be a $\tau$-finite spectral  projection of $|x_2|$ such that $\norm{|x_2|({\bf 1}-p)}_{L_1\cap L_\infty }\le \varepsilon$.
Hence, we obtain that
\begin{align*}
|\tau(T(x_1) |x_2| u_2^*  )|&\stackrel{\eqref{0F}}{=}
 |\tau(T(x_1)  ({\bf 1}-p) |x_2| u_2^* )|\\
&~\le \norm{T(x_1)}_{L_1+L_\infty } \norm{ ({\bf 1}-p) |x_2| u_2^* }_{L_1\cap L_\infty } \\
&~\le  \norm{T(x_1)}_{L_1+L_\infty } \cdot \varepsilon.
\end{align*}
Since $\varepsilon$ is arbitrarily taken, it follows that $|\tau(T(x_1) x_2^*  )| =0$.
\end{proof}

The following result is a semifinite version of \cite[Corollary 1]{Sourour}.
\begin{cor}\label{cor:disjoint}
Let $E(\cM,\tau)$ be a symmetric   space having order continuous norm and $\norm{\cdot}_E$ is not proportional to $\norm{\cdot}_2$.
Let $T$ be a bounded hermitian operator on $E(\cM,\tau)$.
For any $x\in E(\cM,\tau )$,
there exist    $y,z\in E(\cM,\tau)$ such that $T(x)=y+z$ and $r(y)\le r(x)$ and $l(z)\le l(x)$.

\end{cor}
\begin{proof}
Denote $A:=T(x) $. Note that
$$
A=l(x)Ar(x)+  l(x) A r(x)^\perp   + l(x)^\perp  Ar(x) +(l(x)^\perp A r(x)^\perp ).
$$
Assume that $l(x)^\perp A r(x)^\perp\ne 0$.
Let
$p\in \cF(\tau )$ be a $\tau$-finite projection such that
  $z= l(x)^\perp A r(x)^\perp p \ne 0$.
Let $zp =u |zp|$ be the polar decomposition.
Then,
$$
u^*l(x)^\perp A r(x)^\perp p= u^* z p \ge 0,
$$
i.e.,
$$\tau  (T(x ) r(x)^\perp pu^* l(x)^\perp  )= \tau  (u^*l(x)^\perp A r(x)^\perp p ) > 0 $$
Note that $l\left(l\left(x\right)^\perp u p r(x)^\perp \right) \perp l(x) $  and $r\left(l(x)^\perp u p r(x)^\perp \right) \perp r(x) $.
By Lemma \ref{lemma:orthogonal} above, we obtain that
$$\tau  (u^*l(x)^\perp A r(x)^\perp p )=\tau  ( T(x ) r(x)^\perp p u^*l(x)^\perp )=  0 ,$$
 which is a contradiction.
 Taking $y=l(x) Ar(x) +l(x) ^\perp A r(x)$ and $z= l(x) A r(x)^\perp$, we complete the proof (note that the choices of $y$ and $z$ are not necessarily unique).
\end{proof}


The following lemma shows that any bounded  hermitian operator $T$ on $E(\cM,\tau)$ (whose norm is not proportional to $\norm{\cdot}_2$) maps the set of all  $\tau$-finite projections to a uniformly  bounded set in $\cM$, which should be compared with the estimates in \cite[Remark 2.5]{Sukochev}.
\begin{lemma}\label{lemmaptop}
Let $E(\cM,\tau )$
is an arbitrary symmetric operator space having order continuous norm.
Assume that $\norm{\cdot}_E$ is not proportional to $\norm{\cdot}_2$.
Let $T$ be a bounded hermitian operator on  $E(\cM,\tau)$.
Then, $\norm{T(p)}_\infty \le 3 \norm{T}$ for any $\tau$-finite projection $p\in \cP(\cM)$.
\end{lemma}
\begin{proof}

 Let  $p\in \cP(\cM)$ be an arbitrary    $\tau$-finite projection.
By Corollary \ref{cor:disjoint} above, we have
\begin{align}\label{eq:defTp}
T(p)=A_{p} p  +p B_{p }
\end{align}
for some $A_p,B_p\in E(\cM,\tau)$ with $r(A_p)\le p$ and $l(B_p)\le p$.
We note that the choices of $A_p$ and $B_p$ are not necessarily unique.
For two $\tau$-finite projections $q_1 ,q_2\in \cM$ such that $q_1q_2=0$ and   $q_1+q_2=p$, we have $$
T(p)=
T(q_1+q_2) = A_{q_1} q_1 +q_1 B_{q_1} +A_{q_2} q_2 +q_2 B_{q_2}  =
  (A_{q_1} q_1 +A_{q_2} q_2)p  +p (q_1 B_{q_1}  +q_2 B_{q_2}).$$
Therefore, we have \begin{align}\label{Tpdec}q_1T(q_1)q_1\stackrel{\eqref{eq:defTp}}{=} q_1A_{q_1}q_1 +q_1B_{q_1}q_1 = q_1T(p)q_1 \end{align}
and
 \begin{align}\label{Tpdec2}p^\perp A_p q_1\stackrel{\eqref{eq:defTp}}{=}   p^\perp T(p)q_1 & = p^\perp (A_{q_1} q_1 +A_{q_2}q_2)p q_1 = p^\perp A_{q_1} q_1 \nonumber  \\
 &=p^\perp \left( A_{q_1} q_1  +q_1 B_{q_1 } \right)\stackrel{\eqref{eq:defTp}}{=} p^\perp  T(q_1) . \end{align}

Consider the following decomposition (see Corollary \ref{cor:disjoint} or \eqref{eq:defTp})
$$T(p) = p^\perp  T(p)+ T(p)p^\perp +p T(p)p=p^\perp A_p p +p B_pp^\perp+pT(p)p.$$
We claim that $\norm{p^\perp A_p  }_\infty \le \norm{T}_{E\to E}$.
Assume by contradiction that $\norm{p^\perp A_p }_\infty > \norm{T}_{E\to E}$.
Then,
taking $q=E^{|p^\perp A_p|}(\norm{T},\infty)\ne 0$, we obtain that
$$\norm{T}  \norm{   q   }_E <   \norm{p^\perp A_ p    q    }_E \stackrel{\eqref{Tpdec2}}{=} \norm{p^\perp  T(q )      }_E\le \norm{ T( q)      }_E \le \norm{T }  \norm{ q}_E   ,$$
which is a contradiction.
Arguing similarly, $\norm{ B_p p^\perp }_\infty \le \norm{T}$.
Now, we aim to show that $\norm{pT(p) p}_\infty \le \norm{T}$.
Note that, by \eqref{eq:defTp}, we have
$$pT(p) p  =p (A_p+B_p)p. $$
 Let $C:= p(A_p+B_p)p $.
We claim that $C$ is self-adjoint.
Indeed, assume that $p (A_p+B_p)p =a+ib$, where $a,b$ are non-zero self-adjoint operators in $E(\cM,\tau)$ and $r(a),r(b)\le p$.
Let $q:= E^{b}(0,\infty)\le p$  (or $q:=E^{b }(-\infty,0)$ if $E^{b }(0,\infty)=0$) so that $qbq\ne 0$, which is  a $\tau$-finite projection.
  We obtain that
  $$\tau(T(q)q)=\tau(qT(q)q)\stackrel{\eqref{Tpdec}}{=}   \tau(qp T(p)pq)\stackrel{\eqref{eq:defTp}}{=}  \tau(q p(A_p+B_p)p q)=\tau(qaq+iqbq)\notin \mathbb{R} ,$$
   which contradicts  Proposition \ref{prop:tppR}.
  Hence, $b=0$.
  This implies that $pT(p)p$ is self-adjoint.
Recall that
for any $q\le p$, we have
 $$q T(p) q \stackrel{\eqref{Tpdec}}{=} q T(q)q.  $$
Let $q:= E^{|pT(p)p |}[\norm{T}+\varepsilon,\infty)$, $\varepsilon>0$.
If $q\ne 0$, then we have
 $$
\left(\norm{T}+\varepsilon\right) \norm{q}_E =\norm{ \left(\norm{T}+\varepsilon\right)   q}_E \le \norm{ p T(p)p q}_E  =\norm{qT(p)q}_E  =\norm{q T( q)p }_E \le \norm{T(  q)}_E\le \norm{T} \norm{q}_E,
 $$
 which is a contradiction.
 Hence, $E^{|pT(p)p |}[\norm{T}+\varepsilon,\infty)=0$ for any $\varepsilon>0$.
 This implies that $\norm{pT(p)p}_\infty \le \norm{T}$.

 Combining the estimates  $\norm{p^\perp A_p  }_\infty \le \norm{T} $, $\norm{ B_p p^\perp }_\infty \le \norm{T}$  and  $\norm{pT(p)p}_\infty \le \norm{T}$, we obtain that
  $$\norm{T(p)}_\infty \le 3 \norm{T}$$ for any $\tau$-finite projection $p$.
In particular, $A_p$ and $B_p$ can be taken such that $\norm{A_p}_\infty,\norm{B_p}_\infty \le 3\norm{T} $ (see the proof for  Corollary \ref{cor:disjoint}).
\end{proof}

The following lemma is the key auxiliary tool in the proof of Proposition \ref{prop:general:H} below, which shows that any bounded hermitian operator $T$ on $E(\cM,\tau)$ is  a bounded operator from $(\cF(\tau),
\norm{\cdot}_\infty)$  into $(C_0(\cM,\tau),
\norm{\cdot}_\infty)$.
By applying the generalized Gleason theorem\cite[Theorem 5.2.4]{Hamhalter} (see also \cite{BJM,Mori,GS1,GS2} for related results), we succeed to prove all but one case
 in the following lemma.
However, one should note that the case when   a von Neumann algebra has $I_2$ direct summand is  exceptional, which can not be covered by the generalized  Gleason theorem. This special case  is rather complicated  and requires
careful study of the restriction of a  hermitian operator on the type $I_2$ summand.

We denote by $C_0(\cM,\tau)$ the closure in the norm $\norm{\cdot}_\infty $ of the linear span of $\tau$-finite projections in $\cM $.
Equivalently, $C_0(\cM,\tau)=\{a\in S(\cM,\tau): \mu(a)\in L_\infty (0,\infty ), ~\mu(\infty,a)=0\}$\cite[Lemma 2.6.9]{LSZ}.

\begin{lemma}\label{lemmaMtoM}
Let $E(\cM,\tau )$
is an arbitrary symmetric operator space having order continuous norm.
Assume that $\norm{\cdot}_E$ is not proportional to $\norm{\cdot}_2$.
Let $T$ be a bounded  hermitian operator on  $E(\cM,\tau)$.
Then, $T$ is a bounded operator from $(\cF(\tau),
\norm{\cdot}_\infty)$  into $(C_0(\cM,\tau) ,
\norm{\cdot}_\infty)$.
In particular, $T$ extends to  a bounded operator from  $C_0(\cM,\tau)$ into $C_0(\cM,\tau)$.
\end{lemma}
\begin{proof}
There exists a decomposition $\cM = \cM_1 \oplus \cM_2$, where $\cM_1$ has no type $I_2$ direct summand and $\cM_2$ is either $0$ or the  type $I_2$ direct summand of $\cM$ (see e.g. \cite[Chapter III.1.5.12]{Blackadar}).

By Lemma \ref{lemmaptop},   $\norm{T(p)}_\infty\le 3\norm{T}$ for any $\tau$-finite projection $p\in \cM_1$.
Moreover, $T(p) =A_p+B_p$ (see Corollary \ref{cor:disjoint}) for some $A_p,B_p\in E(\cM,\tau)$ with $r(A_p)\le p$ and $l(B_p)\le p$.
Since $p\le {\bf 1}_{\cM_1}$, it follows that $A_p =A_p {\bf 1}_{\cM_1}\in \cM_1$ and $B_p = {\bf 1}_{\cM_1} B_p\in \cM_1$.
On the other hand, $\tau(p)<\infty $  implies that  $T(p)\in \cF(\cM_1)\subset C_0(\cM_1,\tau)\subset C_0(\cM,\tau)$.
It follows from \cite[Theorem 5.2.4]{Hamhalter} that for any $\tau$-finite projection $p$,  $T|_{\cP_f(\cM_1 )}$
extends uniquely to a bounded linear operator from the  reduced algebra $p\cM p$ into $C_0(\cM,\tau)$.
 We denote this operator by $R_p$.
 Moreover, $$\norm{R_p|_{p\cM_1  p}}_{\norm{\cdot}_\infty \to \norm{\cdot}_\infty } \le 12 \norm{T}$$(see the proof of\cite[Theorem 5.2.4]{Hamhalter}).
 Moreover, by the uniqueness of the  extension $R_p$\cite[Theorem 5.2.4]{Hamhalter}, we obtain that if $p\ge q$, then the extension $R_p$ of $T$  coincides with $R_q$ on $q\cM_1q$.

 We claim that $R_p$ coincide with $T$ on $p \cM_1 p$.
 Indeed, let $x$ be a positive operator in $p\cM_1  p$.
 Then, there exists a sequence of positive operators $x_n$
whose singular values are step functions such that $x_n\uparrow x$ and $\norm{x_n-x}_\infty \to 0$.
Since $R_p$ coincides with $T$ on all projection $q\le p$, it follows that $R_p(x_n)=T(x_n)$.
Since $E(\cM_1 ,\tau)$ has order continuous norm, it follows that $T(x_n)\to T(x)$ in $\norm{\cdot}_E$, and therefore, in the measure topology\cite[Proposition 20]{DP2014}.
On the other hand,
the $\norm{\cdot}_\infty$-$\norm{\cdot}_\infty$-boundedness of $R_p$ implies that $R_p(x_n)\to R_p(x)$ in $\norm{\cdot}_\infty$, and therefore, in the measure topology\cite[Proposition 20]{DP2014}.
Hence, $T(x)=R_p(x)$.
Since $p$ is an arbitrary $\tau$-finite projection, it follows that $T$ is a bounded linear  operator from $(\cF(\tau)\cap \cM_1,
\norm{\cdot}_\infty)$  into $(  C_0(\cM_1 ,\tau),
\norm{\cdot}_\infty)$.

Now, we consider the case when $\cM_2$ is   a non-vanishing  type $I_2$ von Neumann direct summand (if $\cM_2=0$, then the lemma follows from the above result).
It is known that $\cM_2$ can be written as $\mathbb{M}_2 \overline{\otimes} \cA $, where $\mathbb{M}_2$ is the algebra of all $2\otimes 2$ matrices and $\cA$ is a $\sigma$-finite commutative von Neumann algebra  (see e.g. \cite{Tak,KR} or \cite[Chapter III.1.5.12]{Blackadar}).
For every element in the form of
$ \left(
    \begin{array}{cc}
      p & 0 \\
      0 & 0 \\
    \end{array}
  \right)
 $, where $p$ is a projection in $\cA$ such that $\tau({\bf 1}\otimes p)<\infty $, we have (see Corollary \ref{cor:disjoint})
 $$T\left(
    \begin{array}{cc}
      p & 0 \\
      0 & 0 \\
    \end{array}
  \right)=  A_p \left(
    \begin{array}{cc}
      p & 0 \\
      0 & 0 \\
    \end{array}
  \right)+\left(
    \begin{array}{cc}
     p & 0 \\
      0 & 0 \\
    \end{array}
  \right)B_p.$$
By Lemma \ref{lemmaptop}, $A_p$ and $B_p$ are uniformly bounded. Assume that
$ A_p= \left(
    \begin{array}{cc}
      a_1 & a_2 \\
      a_3 & a_4 \\
    \end{array}
  \right)$ and $ B_p= \left(
    \begin{array}{cc}
      b_1 & b_2 \\
      b_3 & b_4 \\
    \end{array}
  \right)$.
Without loss of generality, we may assume, in addition, that   $a_2,a_4,b_3,b_4$ are  $0$.
  Hence,
 $$T\left(
    \begin{array}{cc}
      p & 0 \\
      0 & 0 \\
    \end{array}
  \right)=   \left( \begin{array}{cc}
      a_1 & 0 \\
      a_3 & 0 \\
    \end{array} \right)\left(
    \begin{array}{cc}
      p & 0 \\
      0 & 0 \\
    \end{array}
  \right)+\left(
    \begin{array}{cc}
      p & 0 \\
      0 & 0 \\
    \end{array}
  \right)  \left( \begin{array}{cc}
      b_1 & b_2 \\
      0 & 0 \\
    \end{array} \right)=   \left( \begin{array}{cc}
      (a_1+b_1)p  & b_2 p \\
      a_3 p & 0 \\
    \end{array} \right) .$$
  Recall that $\norm{A_p}_\infty ,\norm{B_p}_\infty \le 3\norm{T}$ (see the proof of Lemma \ref{lemmaptop}).
  We obtain that $a_1,a_3,b_1,b_2\le 3\norm{T}$.
  Hence, for any   $x\in \cA$ whose singular value function is a step function, we have
  $$ T\left(
    \begin{array}{cc}
     x & 0 \\
      0 & 0 \\
    \end{array}
  \right)= \left(
    \begin{array}{cc}
    h_1 x &  h_2 x \\
     h_3 x & 0 \\
    \end{array}
  \right),$$
 where $\norm{h_1}_\infty ,\norm{h_2}_\infty ,\norm{h_3}_\infty \le 6\norm{T}$.
 Therefore, $\norm{T\left(
    \begin{array}{cc}
     x & 0 \\
      0 & 0 \\
    \end{array}
  \right)}_\infty \le 12 \norm{T}\norm{x}_\infty $.
 For any self-adjoint  $x\in \cF(\tau)$ there exists a sequence of self-adjoint elements  $x_n\in \cA$ such that $|x_n|\uparrow |x|$,  $\norm{x_n-x}_\infty \to 0$ and   $\mu(x_n)$ are step functions.
 Since $E(\cM,\tau )$ has order continuous norm, it follows that $$\norm{T\left(
    \begin{array}{cc}
     x_n & 0 \\
      0 & 0 \\
    \end{array}
  \right) -T\left(
    \begin{array}{cc}
     x & 0 \\
      0 & 0 \\
    \end{array}
  \right)}_E \to 0 $$
and therefore,  $T\left(
    \begin{array}{cc}
     x_n & 0 \\
      0 & 0 \\
    \end{array}
  \right) \to T\left(
    \begin{array}{cc}
     x & 0 \\
      0 & 0 \\
    \end{array}
  \right)$ in measure \cite[Proposition 20]{DP2014}.
On the other hand,
 $\norm{T\left(
    \begin{array}{cc}
     x_n  & 0 \\
      0 & 0 \\
    \end{array}
  \right)}_\infty \le 12 \norm{T}\norm{x_n }_\infty\le 12 \norm{T}\norm{x}_\infty $.
  Since the unit ball of  $\cM_2$ is closed in the measure topology\cite[Theorem 32]{DP2014}, we obtain that
   $$\norm{T\left(
    \begin{array}{cc}
     x  & 0 \\
      0 & 0 \\
    \end{array}
  \right)}_\infty \le 12 \norm{T}\norm{x}_\infty $$
   for any self-adjoint operator $\left(
    \begin{array}{cc}
     x & 0 \\
      0 & 0 \\
    \end{array}
  \right)\in \cF(\tau)\cap \cM_2$.
  Since every element in $\cF(\tau)\cap \cM_2$ is the combination of  two self-adjoint elements, we obtain that $$\norm{T\left(
    \begin{array}{cc}
     x & 0 \\
      0 & 0 \\
    \end{array}
  \right)}_\infty \lesssim \norm{T}\norm{x}_\infty $$ for any operator $\left(
    \begin{array}{cc}
     x & 0 \\
      0 & 0 \\
    \end{array}
  \right)\in \cF(\tau)\cap \cM_2 $.
 The same argument show that
 $\norm{T\left(
    \begin{array}{cc}
     0& 0 \\
      0 & x \\
    \end{array}
  \right)}_\infty \lesssim \norm{T}\norm{x}_\infty $.
 For estimates of $\norm{T\left(
    \begin{array}{cc}
     0& 0 \\
      x & 0\\
    \end{array}
  \right)}_\infty$ and $\norm{T\left(
    \begin{array}{cc}
     0& x \\
      0 & 0\\
    \end{array}
  \right)}_\infty$, one only need to note that
  $$T'(\cdot ):= \left( \left(
    \begin{array}{cc}
     0& {\bf 1_\cA}   \\
      {\bf 1_\cA}  & 0\\
    \end{array}
  \right) + {\bf 1_{\cM_1}}\right) T  \left( \left( \left(
    \begin{array}{cc}
     0& {\bf 1_\cA}   \\
      {\bf 1_\cA}  & 0\\
    \end{array}
  \right) + {\bf 1_{\cM_1}}\right) \cdot \right) $$
 is also a hermitian operator on $E(\cM,\tau)$ and $$T'\left(
    \begin{array}{cc}
     x & 0 \\
      0 & 0 \\
    \end{array}
  \right)
  =T\left(
    \begin{array}{cc}
    0 & 0 \\
      x & 0 \\
    \end{array}
  \right).
  $$
 By taking the linear combination, we obtain that  for any $x\in \cF(\tau)\cap \cM_2$, $$\norm{T(x)}_\infty \lesssim \norm{T}\norm{x}_\infty .$$
This completes the proof.
\end{proof}

We prove below an analogue of Proposition \ref{prop:tppR}  for the
symmetric space $C_0(\cM,\tau)$.

\begin{proposition}\label{redu}
 Let $E(\cM,\tau )$
is an arbitrary symmetric operator space having order continuous norm.
Assume that $\norm{\cdot}_E$ is not proportional to $\norm{\cdot}_2$.
 Let $T$ be a bounded  hermitian operator on    $E(\cM,\tau)$.
Then, for any operator $x\in C_0(\cM,\tau)$ and
 a $\tau$-finite projection  $p\in \cP(\cM)$ commuting with $|x|$,
we have
$$\langle Tx ,  pu^* \rangle_{(C_0(\cM,\tau ),C_0(\cM,\tau )^*)}:=\tau(T(x) pu^*  )\in \mathbb{R}, $$
where $x=u|x|$ is the polar decomposition.
\end{proposition}
\begin{proof}
We only consider the case when $x$ is positive. The
proof of  the general case follows from the same argument by replacing Proposition \ref{prop:tppR} used below with  Corollary \ref{prop:tUUR}.

Let $x_n:= \sum_{1\le k\le n} \alpha_k p_k \in \cF(\tau  ) $  be such that $x_n\to x$ in $\norm{\cdot}_\infty $, where $p_k$ are $\tau$-finite spectral projections of $x_n$ which commute    with $p$,  and $\alpha_k$ are real numbers. For each  $p_k$, we have
\begin{align*}
  & \langle Tp_k, p\rangle_{(C_0(\cM,\tau ),C_0(\cM,\tau )^*)}\\
= ~&\langle T(p p_k), p p_k\rangle_{(C_0(\cM,\tau ),C_0(\cM,\tau )^*)} +\langle T(pp_k), p-pp_k\rangle_{(C_0(\cM,\tau ),C_0(\cM,\tau )^*)}  \\
& ~+  \langle T(p_k -p p_k ), p  \rangle_{(C_0(\cM,\tau ),C_0(\cM,\tau )^*)} \\
\stackrel{\eqref{lemma:orthogonal} }{=}  & \langle T(p p_k), p p_k\rangle_{(C_0(\cM,\tau ),C_0(\cM,\tau )^*)}\\
~ = ~ &
\tau(T(p p_k)  p p_k)
\stackrel{\eqref{prop:tppR}}{  \in} \mathbb{R}.
\end{align*}
Hence, $\langle Tx_n , p\rangle_{(C_0(\cM,\tau ),C_0(\cM,\tau )^*)} \in \mathbb{R}$ for every $n$.
Moreover, we have
$$|\langle  Tx,p \rangle_{(C_0(\cM,\tau ),C_0(\cM,\tau )^*)} -\langle  Tx_n ,p \rangle_{(C_0(\cM,\tau ),C_0(\cM,\tau )^*)} | \le
\norm{T}_{C_0(\cM,\tau)\to C_0(\cM,\tau) }\norm{x-x_n}_\infty  \norm{p}_1\to 0,$$
which shows that $\langle Tx , p\rangle_{(C_0(\cM,\tau ),C_0(\cM,\tau )^*)}\in \mathbb{R}$.
\end{proof}

\begin{proposition}\label{prop:general:H}
Let $E(\cM,\tau )$
is an arbitrary symmetric operator space having order continuous norm.
Assume that $\norm{\cdot}_E$ is not proportional to $\norm{\cdot}_2$.
 Let $T$ be a bounded  hermitian operator on   $E(\cM,\tau)$.
Then, $T$ can be extended to a bounded operator on  $C_0(\cM,\tau)$ (still denoted by $T$) and, for any
  operator $x\in C_0(\cM,\tau )$,  there exists a support functional  $x'$ in $C_0(\cM,\tau)^*$ of $x$ such that
 $\langle Tx, x' \rangle \in \mathbb{R}$.
  In particular, $T$ is a hermitian operator on $C_0(\cM,\tau)$.
\end{proposition}
\begin{proof}
By Lemma \ref{lemmaMtoM}, $T$ can be extended to a bounded operator on  $C_0(\cM,\tau)$.

Without loss of generality, we assume, in addition, that $\norm{x}_\infty =1$.
Let $x=u|x|$ be the polar decomposition.
Recall that $x\in C_0(\cM,\tau)$.
Hence, $\tau(E^{|x|} (1-\frac1n, 1])<\infty $ for any $n>0$.
Recall that  $(C_0(\cM,\tau))^\times =L_1(\cM,\tau)$ (see e.g. \cite[Lemma 8]{SS14} and \cite[Theorem 53]{DP2014}).
Define $x_n = \frac{E^{|x|}(1-\frac1n, 1] }{\tau(E^{|x|} (1-\frac1n, 1]) } u^*\in L_1(\cM,\tau)\subset C_0(\cM,\tau)^*$, $n\ge 1$.
We have  $\norm{x_n }_{C_0(\cM,\tau)^*} =\norm{x_n}_1=1$ \cite[p.228]{DP2014}.   
Note that
\begin{align}\label{net1}
1-\frac1n\le  \tau(xx_n  )=\frac{\tau(|x| E^{|x|}(1-\frac1n, 1]) }{\tau(E^{|x|} (1-\frac1n, 1]) } \le 1.
\end{align}
By Alaoglu's theorem\cite[p.130, Theorem 3.1]{Conway}, there  exists a subnet $\{x_i\}$ of $\{x_n\}_n$ converging  to some element $x'\in C_0(\cM,\tau)^*$
in the weak$^*$ topology of   $C_0(\cM,\tau)^*$
 and $\norm{x'}_{C_0(\cM,\tau) ^* }\le 1$.
On the other hand, we have
$$\norm{x'}_{C_0(\cM,\tau)^*}= \norm{x}_{\infty } \norm{x'}_{C_0(\cM,\tau)^*}  \ge  x' ( x  )=\lim _i \tau(xx_i ) \stackrel{\eqref{net1}}{=}1. $$
Hence, $\norm{x'}_{C_0(\cM,\tau)^*}=1$.
This implies that $x'$ is a support functional of $x$.
Therefore,
by taking $p=E^{|x|} (1-\frac1n, 1] $ in   Proposition \ref{redu}, we obtain that
$$  \langle Tx, x' \rangle=(w^*)-\lim_i  \langle Tx, x_i  \rangle \in \mathbb{R }. $$
This completes the proof.
\end{proof}



Recall that a derivation $\delta$ on an algebra $\cA$ is a linear operator satisfying the Leibniz rule.
Although it is known that a  derivation from $C_0(\cM,\tau)$ into $C_0(\cM,\tau)$ is not necessarily inner\cite{BHLS,Huang} (see   \cite[Example 4.1.8]{Sakai} for  examples of    non-inner derivations   on $K(\cH)$), it is shown recently that every derivation $\delta$ from an arbitrary  von Neumann subalgebra of $\cM$ into $C_0(\cM,\tau)$ is inner, i.e., there exists an element  $a\in C_0(\cM,\tau)$ such that $\delta(\cdot)=[a,\cdot]$\cite{BHLS2,Huang}.
 On the other hand,
 every    derivation from $C_0(\cM,\tau)$ into $C_0(\cM,\tau)$ is spatial, i.e., it can be implemented by an element from $\cM$ (see e.g. \cite[Theorem 2]{Sakai} and \cite[Theorem 4.1]{Arveson}).

\begin{lemma}\label{derivation}Every derivation $\delta$ from $C_0(\cM,\tau)$ into $C_0(\cM,\tau)$ is spatial.
In particular, if $\delta$ is a $*$-derivation, then the element implementing $\delta$ can be chosen to be self-adjoint.
\end{lemma}
\begin{proof}
The first statement follows from  \cite[Theorem 2]{Sakai} (or \cite[Theorem 4.1]{Arveson}) and the fact that $C_0(\cM,\tau)$ is a $C^*$-algebra.
For the second statement, see e.g. \cite[Chapter 3.4, Remark 3.4.1]{Huang}.
\end{proof}

We now come to the main result of this section, which gives the full  description of hermitian operators   on a  symmetric space $E(\cM,\tau)$.

\begin{theorem}\label{th:her}

Let $E(\cM,\tau)$ be a   symmetric  space affiliated with  an atomless semifinite von Neumann algebra (or an atomic von Neumann
algebra with all atoms having the same trace) $\cM$  equipped
with a semifinite faithful normal trace $\tau$.
Assume that $\norm{\cdot}_E$ is order continuous and is not proportional to $\norm{\cdot}_2$.
Then, a bounded linear operator $T$ on $E(\cM,\tau)$ is a hermitian operator on $E(\cM,\tau)$ if and only if  there exist  self-adjoint operators $a$ and $b$ in $\cM$ such that
\begin{align}\label{eq:her}
Tx=ax+xb,~x\in E(\cM,\tau).
\end{align}
In particular, $T$ can be extended to a bounded hermitian operator on  the von Neumann algebra $\cM$.
\end{theorem}
\begin{proof}
The `if' part of the theorem  is obvious (see e.g. the argument in  \cite[p.71]{Sourour} or \cite[p. 167]{FJ2}).

By Corollary \ref{prop:general:H}, $T$ is a bounded  hermitian operator on $C_0(\cM,\tau)$.
Recall that any hermitian operator $T$ on a $C^*$-algebra $\cA$ is the sum of  a left-multiplication by a self-adjoint operator in $\cA$ and
a $*$-derivation in $\cA$ (see e.g. \cite[p.213]{Sinclair}).
It follows from  Lemma \ref{derivation}  that there exist self-adjoint elements $a,b\in \cM$ such that
$$Tx=ax+xb,~x\in C_0(\cM,\tau).$$
Noting that $\cF(\tau)\subset C_0(\cM,\tau)$, we obtain that
$$Tx=ax+xb,~x\in \cF(
\tau) .$$
Since $E(\cM,\tau)$ has order continuous norm, it follows that  $\cF(\tau)$ is dense in $(E(\cM,\tau),\norm{\cdot}_E)$ (see e.g. \cite[Proposition 46]{DP2014} or \cite[Remark 2.9]{HSZ}).
For any $x\in E(\cM,\tau)$, there exists a sequence  $\{y_n\}\subset \cF(\tau)$ such that $\norm{y_n-x}_E\to 0$.
Hence,
we obtain that
$$Tx= \norm{\cdot}_F-\lim_n T(y_n) = \norm{\cdot}_F-\lim_n \left(ay_n +y_n b\right )=  ax+xb,~x\in E(\cM,\tau) .$$
This completes the proof.
\end{proof}

\section{Isometries}
The goal of this section is to answer   the question posed in \cite{Sukochev,CMS} and stated at the outset of this paper.
Throughout this section, unless stated otherwise, we always assume that $\cM$  is an atomless semifinite
von Neumann algebra or an atomic semifinite von Neumann algebra with all atoms having the
same trace, and we assume that $\tau$ is a semifinite faithful normal trace on $\cM$.

Before proceeding to the proof of Theorem \ref{th:iso}, we need  the following auxiliary tool, which  extends \cite[Corollary 2]{Sourour} and \cite[Corollary 3.2]{Sukochev}.

 \begin{cor}\label{4.2}Let $(\cM,\tau )$ be an atomless semifinite von Neumann algebra or an atomic von Neumann algebra whose atoms having the same trace.
 Let $E(\cM,\tau)$ be a   symmetric operator space  whose norm is order continuous and  is not proportional to $\norm{\cdot}_2$.
Let $T$ be a bounded hermitian operator on $E(\cM,\tau)$. 
Then, $T^2$ is also a hermitian operator on $E(\cM,\tau)$ if and only if $$T(y)=a   y+y   b  ,~\forall y\in L_1(\cM,\tau )\cap \cM ,$$
for some   self-adjoint operators   $a \in \cM_w $
 and $b\in \cM_{{\bf 1}-w}$, where $w \in P(Z(\cM))$.

\end{cor}
\begin{proof}
$(\Leftarrow)$.
 Note that $T^2(y)= a^2   y+y   b^2  $, ~$\forall y\in L_1(\cM,\tau)\cap \cM$.
It follows from Theorem \ref{th:her} that $T^2$ is a hermitian operator.

$(\Rightarrow)$. Recall that, by Theorem \ref{th:her}, we have $T(y)=ay+yb$, $y\in \cM$, for some self-adjoint elements $a,b\in \cM$.
 Due to the assumption that $T^2$ is also  hermitian, there exist self-adjoint operators $c,d\in \cM$ such that
 \begin{align*}
 T^2(y)= cy+yd   = a^2y + 2 ayb +  yb^2, ~\forall y\in \cM.
 \end{align*}
 The claim follows from Theorem \ref{central decomposition th}.
\end{proof}
\begin{rem}\label{rem42}
Let $T$, $a$, $b$, $w$ be defined as in Corollary \ref{4.2}.
In particular,   $l(a)\le w $ and $l(b)\le {\bf 1} - w $.
Define
$$z_a: = \sup \left\{ p\in \cP(Z(\cM)): p\le w, ~  pa  \in Z(\cM) \right  \}$$
and
$$z_b: = \sup \left\{ p\in \cP(Z(\cM)): p\le {\bf 1}- w , ~  pb  \in Z(\cM) \right  \}.$$
For any elements $z_1,z_2 \in \left\{ p\in \cP(Z(\cM)): p\le w , ~  pa  \in Z(\cM) \right  \} $ and $d \in \cM$, we have
\begin{align*}
(z_1\vee z_2)a d&= (z_1 +z_2 -z_1z_2)a d = z_1 ad + z_2 ({\bf 1}-z_1)ad \\
&= z_1 a  d + z_2 a ({\bf 1}-z_1)d = d z_1a +  d ({\bf 1}-z_1)  z_2 a   =d  (z_1\vee z_2) a.
\end{align*}
That is, $z_1\vee z_2 \in \left\{ p\in \cP(Z(\cM)): p\le w , ~  pa  \in Z(\cM) \right  \} $.
Hence, $\left\{ p\in \cP(Z(\cM)): p\le w , ~  pa  \in Z(\cM) \right  \} $ is an increasing  net with  the partial order $\le $ of projections.
Therefore,   by Vigier's theorem \cite[Theorem 2.1.1]{LSZ}, we obtain that
$z_a a \in Z(\cM)$.
That is,
 $$z_a \in \left\{ p\in \cP(Z(\cM)): p\le w, ~  pa  \in Z(\cM) \right  \} .$$
 Arguing similarly,
  $$z_b  \in \left\{ p\in \cP(Z(\cM)): p\le {\bf 1}-w , ~  pb  \in Z(\cM) \right  \} .$$
  We have
\begin{align}\label{remarkequ}
T(y)=a (w-z_a)  y+y  b(({\bf 1}-w) -z_b) + (az_a +bz_b)  y ,~y\in L_1(\cM,\tau )\cap \cM .
\end{align}
In particular,
\begin{enumerate}
  \item $w-z_a$, $({\bf 1}-w) -z_b$, $z_a+z_b$ are
pairwise orthogonal projections in $Z(\cM)$;
  \item if $p\in Z(\cM)$ such that $p\le w-z_a$ and  $ap\in Z(\cM,\tau)$ (or $p\le ({\bf 1}-w) -z_b$ and $bp\in Z(\cM)$), then  $p=0$;
  \item   $w-z_a$ (resp., $ ({\bf 1}-w) -z_b $) is the central support of $a (w-z_a)$ (resp.,   $   b(({\bf 1}-w ) -z_b)$).
\end{enumerate}
\end{rem}

\begin{rem}
Let $\cM$ be a     semifinite factor.
It is an immediate consequence of Corollary \ref{4.2} that
if  $T$  a bounded hermitian operator on  $\cM$,
then  $T^2$ is hermitian if and only if $T$ is a left(or right)-multiplication by a self-adjoint operator in $\cM$
(see also the proof of \cite[Corollary 2]{Sourour} and \cite[Corollary 3.2]{Sukochev}).
\end{rem}

Let $\cM_1$ and $\cM_2$ be two von Neumann algebras.
A  complex-linear  map $J:\cM_1 \stackrel{ \tiny injective }{\longrightarrow} \cM_2$ is called Jordan $*$-isomorphism if  $J(x^*)=J(x)^*$ and $J(x^2)=J(x)^2$, $x\in \cM_1$ (equivalently, $J(xy+yx)=J(x)J(y)+J(y)J(x)$, $x,y\in \cM_1$) (see e.g. \cite{Yeadon,Sherman,BR}).
A   Jordan $*$-isomorphism is called normal if it is completely additive (equivalently, ultraweakly continuous).
Alternatively, we adopt the following equivalent definition: $J(x_\alpha)\uparrow J(x)$ whenever $x_\alpha\uparrow x\in \cM_1^+$ (see e.g \cite[Chapter I.4.3]{Dixmier}).
If $J :\cM_1\to \cM_2$ is a surjective Jordan $*$-isomorphism, then $J$ is necessarily normal\cite[Appendix A]{RR}.

The following theorem is the main result of this section.
Due to the complicated nature of hermitian operators on a von Neumann algebra distinct from  a factor,
the proof below is substantially more involved than those in \cite{Sourour,Sukochev}.

\begin{theorem}\label{th:iso}

Let $\cM_1$ and $\cM_2$ be  atomless von Neumann algebras (or  atomic von Neumann algebras whose atoms have the same trace) equipped with semifinite faithful normal traces $\tau_1$ and $\tau_2$, respectively.
Let $E(\cM_1,\tau_1)$ and $F(\cM_2,\tau_2)$ be two  symmetric operator spaces whose norms are order continuous and  are not proportional to  $\norm{\cdot}_2$.
If $T:E(\cM_1,\tau_1)\to F(\cM_2,\tau_2)$ is a surjective isometry, then there exist
two nets of elements $ A_i \in  F(\cM_2 ,\tau_2),~   i\in I $,  disjointly supported from the right and $ B_i \in F(\cM_2 ,\tau_2), ~i\in I$,    disjointly supported from the  left,   and a  surjective  Jordan $*$-isomorphism $J:\cM_1\to \cM_2$  and a central projection $z\in \cM_2$ such that
$$T(x) =  \norm{\cdot}_F-  \sum _{i\in I}    J(x)A_i z + B_i J(x) ({\bf 1}-z)  , ~x\in E(\cM_1,\tau_1)\cap \cM_1,$$
 where  the series is taken as the limit of  all finite partial sums.

\end{theorem}
\begin{proof}

The indices of von Neumann algebras  $\cM_1$ and $\cM_2$ play no role in the proof below. So, to reduce the notation,   we assume that $\cM_1=\cM_2=\cM$.
We denote by $L_a$ (resp. $R_a$) the left (resp. right) multiplication by $a\in \cM$,
that is, $$L_a(x)=ax$$ (resp. $R_a(x)=xa$) for all $x\in  S(\cM,\tau)$.
For any self-adjoint operator $a\in \cM $,
$TL_aT^{-1}$ and $TL_a^2T^{-1}$ are hermitian on $ F(\cM,\tau)$ (see e.g. \cite[Lemma 2.3]{FJ89} or \cite{JL}).

We divide the proof into several steps.

\textbf{Step 1.} We aim to prove  that there exists  $z\in \cP(Z(\cM)) $ (does not depend on $a$ below) such that
\begin{align}\label{tlt}
TL_aT^{-1}= L_{J_1(a)z }+R_{J_2(a)({\bf 1} -z )},~a=a^*\in \cM,
\end{align}
 where $J_1(a),J_2(a)$ are self-adjoint operators  in $\cM$.
Let $z(f)$ be the central support of $f=f^*\in \cM$.
For any fixed  $b=b^* \in \cM$,
\begin{align}
\label{tlbt1}
TL_b T^{-1}y   \stackrel{\eqref{remarkequ}}{=}L_{J_1(b)} y+ R_{J_2(b)}y  +J_3(b ) y ,~\forall   y\in L_1(\cM,\tau )\cap \cM ,
\end{align}
for some   self-adjoint operators   $J_1(b),J_2(b)\in \cM$, $J_3(b)=J_3(b)^*\in Z(\cM)$
such that
\begin{enumerate}
  \item the $z(J_1(b))$, $z(J_2(b))$ and $z(J_3( b))$ are pairwise orthogonal projections   in $Z(\cM)$ (see (1) and (3)  in   Remark   \ref{rem42});
  \item if $p\in Z(\cM)$ such that $p\le z(J_1(b))$  and  $J_1(b)p\in Z(\cM,\tau)$ (or $p\le z(J_2(b))$ and  $J_2(b)p\in Z(\cM)$),  then $p=0$ (see (2) in  Remark  \ref{rem42}).
\end{enumerate}

By \eqref{tlbt1}, for any self-adjoint $a\in \cM$, we have
\begin{align}
\label{tlbt2} TL_{a } T^{-1} y =  L_{J_1(a  )    } y + R_{J_2 ( a )  } y +J_3(a )y   , ~\forall  y\in L_1(\cM,\tau )\cap \cM \end{align}
and
 \begin{align}\label{tlbt3}
 TL_{a+b} T^{-1} y =  L_{J_1(a+b )    } y +R_{J_2 ( a+ b)  } y +J_3(a+b)y  ,  ~\forall y\in    L_1(\cM,\tau )\cap \cM .
 \end{align}
Now, we consider the reduced algebra $\cM_{z(J_1(b)) \wedge z(J_2(a))}$.
 For all $y\in (L_1\cap L_\infty )(\cM_{z(J_1(b)) \wedge z(J_2(a))},\tau )$,
 we have
\begin{align*}
   L_{J_1(b)} y   + R_{J_2  ( a )  } y \qquad  =\qquad
  &
L_{J_1(b  )    } y + R_{J_2 (b )  } y +J_3(b )y +L_{J_1(a  )    } y + R_{J_2 ( a )  } y +J_3(a )y\\
 \stackrel{\eqref{tlbt1}~\mbox{\tiny  and }\eqref{tlbt2}}{ =}&TL_{b} T^{-1} y +TL_{a } T^{-1} y\\
 = \qquad
   &TL_{b+a } T^{-1} y\\
 \stackrel{\eqref{tlbt3}}{=~}~ \quad  &
  L_{J_1(a+b )    } y + R_{J_2( a+ b)  } y +J_3(a+b)y   .     \end{align*}
By Theorem \ref{central decomposition th},
there exists central projection
$$p\le   z(J_1(b)) \wedge z(J_2(a))$$
 such that
 $$J_1(b)p\mbox{  and  } J_2(a)( z(J_1(b)) \wedge z (J_2(a))  -p)$$
 are in the center $Z(\cM)$ of $\cM$.
 However, by (2) of Remark \ref{rem42} (used twice),
 $$ p=0 =  z(J_1(b))   \wedge z (J_2(a)) -p .$$
That is,
 $$z(J_1(b)) \wedge z (J_2(a)) =0 .$$
 Note that $a,b$ are arbitrarily taken.
Defining
 $$z:=  \bigvee_{b=b^* \in \cM} z (J_1(b)) , $$
   we obtain
  that
  $$  TL_b T^{-1}y  \stackrel{\eqref{tlbt1}}{=}L_{J_1(b)} y+ R_{J_2(b)}y  +J_3(b ) y =  L_{J_1(b) + J_3(b)z} y + R_{J_2(b)+J_3(b)({\bf 1} -z)}y ,~\forall   y\in L_1(\cM,\tau )\cap \cM . $$
  Replacing $J_1(b) + J_3(b)z$ (resp., $J_2(b)+J_3(b)({\bf 1} -z)$) with $J_1(b)$ (resp., $J_2(b)$), we   obtain
    \eqref{tlt}.

\textbf{Step 2.}
Note that $$L_{J_1(a)^2z } +R_{J_2(a)^2 ({\bf 1} -z)} \stackrel{ \eqref{tlt} } {=} (TL_aT^{-1})^2= TL_{a^2}T^{-1} \stackrel{\eqref{tlt}}{= } L_{J_1(a^2)z } +R_{J_2(a^2)({\bf 1} -z)}   $$ for every $a=a^*\in \cM$.
 By standard    argument   (see e.g. \cite[p.117]{Sukochev}),  we obtain
 \begin{align}\label{defJ12}J(\cdot):=J_1 (\cdot) z  +J_2(\cdot )({\bf 1}- z)
 \end{align}
 is an injective Jordan $*$-isomorphism on $\cM$.
Let $0\le a_i\uparrow a\in \cM$.
Clearly, $J(a_i)z \uparrow \le J(a)z $ (see e.g. \cite[Eq.(12)]{HSZ} and \cite[p.211]{BR}).
Since $a_i\uparrow a$, it follows that for any $x\in E(\cM,\tau)$,  $x^*  a_i  x \uparrow x^* a   x $\cite[Proposition 1 (vi)]{DP2014} and $x^*  a_i  x \to x^* a   x $ in measure topology\cite[Proposition 2 (iv)]{DP2014}.
By the fact that     $t\mapsto t^{1/2}$ is an operator monotone function\cite[Proposition 1.2]{DD95},
we obtain that $ \left( x^*  (a- a_i)    x \right) ^{\frac12} $ is a decreasing net.
Note that $\left(x^* (a-a_i )  x \right) ^{\frac12}  \to 0$ in measure\cite[p.213]{DP2014}.
By   \cite[Proposition 2 (iii)]{DP2014}, we obtain that
that $\left(x^* (a-a_i )  x \right) ^{\frac12}  \downarrow 0$ and, therefore,
 $$E(\cM,\tau )\ni |(a-a_i)^{1/2}  x |  = \left(x^* (a-a_i )  x \right) ^{\frac12}  \downarrow 0.$$
It follows from the order continuity of $\norm{\cdot}_E$ that
$$\norm{ J(a-a_i)z T(x) }_F \stackrel{\eqref{tlt}}{ =}  \norm{T((a-a_i)x)}_F=  \norm{(a-a_i)x}_E  \le \norm{(a-a_i)^{1/2}}_\infty \norm{(a-a_i)^{1/2}x}_E \to 0 $$
for all $x\in E(\cM,\tau)$ such that $T(x)$ is a $\tau$-finite projection in $\cF(\tau)$ less than $z$.
Therefore,
$J(a_i)z \to J(a)z$ in localized measure topology \cite[Proposition 20]{DP2014}.
Hence,
$J(a_i)z\uparrow J(a)z$ (see Section \ref{s:p1}).
The same argument shows that $J(a_i)({\bf 1}-z)\uparrow J(a)z({\bf 1}-z) $.
Therefore, $J(\cM)$ is weakly closed\cite[Remark 2.16]{HSZ} and $J:\cM\to J(\cM)$ is a surjective (normal) Jordan $*$-isomorphism.


\textbf{Step 3.}
We claim that $J$ is surjective.
Let $c=c^*\in \cM$.
Note that
$T^{-1}(L_{cz} +R _{c({\bf 1} -z)})T $ and its square $T^{-1}(L_{c^2z} +R _{c^2 ({\bf 1}-z)})T  $
are hermitian operators.
Hence, by Corollary \ref{4.2},
  there exists  a
central projection $z'\in \cM$ and $c'=(c')^*\in \cM$
such that
\begin{align}\label{cform}
T^{-1}(L_{cz} +R _{c({\bf 1}-z)})T = L_{c' z'} + R_{c' ({\bf 1} -z')}.
\end{align}
Employing  the argument used in steps 1 and 2 to  \eqref{cform}, we obtain  a normal injective Jordan $*$-isomorphism $J'':\cM\to \cM$ such that $J''(c) = c'$.
Moreover,
 for each $c\in J(\cM)$, we have
$$L_{J^{-1}(c)}\stackrel{\eqref{tlt}}{=}T^{-1}(L_{cz}+R _{c({\bf 1}-z)})T \stackrel{\eqref{cform}}{=} L_{c' z'} +R_{c' ({\bf 1} -z')}.$$
Hence, $L_{J^{-1}(c)-c'z'}=R_{c' ({\bf 1} -z')}$.
In particular,    $c'({\bf 1} -z')\in Z(\cM)$.
Hence, $L_{J^{-1}(c)} = L_{c'z '+ c'({\bf 1} -z') }=L_{c'}=L_{J''(c)}$.
That is, for each $c\in J(\cM)$, we have
$J''(c) = J^{-1}(c)$ (see also \cite[Remark 3.3]{Sukochev} for the case when $\cM$ is a factor).
Hence, $J''(J(\cM))= J^{-1}(J(\cM))=\cM$.
 By the  injectivity of $J''$, we obtain that $J(\cM)=\cM$.
This proves the claim.

\textbf{Step 4.}
Applying \eqref{tlt} to $T(x)$,
 we obtain that
\begin{align*}
T(ax)& ~=~ J_1(a)z T(x) + T(x)({\bf 1} - z) J_2(a )\\
&~= ~J_1(a)z T(x)  + T(x) J_2(a)({\bf 1} - z) \\
&\stackrel{\eqref{defJ12}}{=}J  (a)z  T(x) +  T(x) J (a)({\bf 1} - z)
 \end{align*}
for all $ a=a^*\in \cM ,~x\in E(\cM,\tau)\cap \cM.$
For any $\tau$-finite projection $e $ in $\cM $, we have
$$T(xe )=J(x)T(e) z + T(e)J(x)({\bf 1}-z), ~x\in \cM . $$
Let  $\{e_i \}_{i\in I}$ be a net of pairwise orthogonal $\tau$-finite projections in $\cM $ such that $\sup_i  e_i ={\bf 1}$\cite[Corollary 8]{DP2014}
and let $\{\lambda_\alpha\}$ be collection  of all finite subsets of $I$, partially  ordered by inclusion.
By the order continuity of $\norm{\cdot}_E$, we have
$\lim_\alpha \sum_{e_\in \lambda_\alpha} xe_i =x$\cite[Theorem 6.13]{DP2012}, and, since $T$ is an isometry, it follows that
\begin{align*}
T(x) &=\norm{\cdot}_F-\lim_\alpha  \sum _{e_i\in \lambda_\alpha} T(x e_i) =\norm{\cdot}_F-\lim_\alpha  \sum _{e_i\in \lambda_\alpha}  J(x)T(e_i) z + T(e_i)J(x)({\bf 1}-z)\\
& =\norm{\cdot}_F-  \sum _{i\in I}  J(x)T(e_i) z + T(e_i)J(x)({\bf 1}-z) ,~\forall x\in E(\cM,\tau)\cap \cM.
\end{align*}

Note that
$$T(e_i) =  J(e_i )T(e_i) z + T(e_i)J(e_i)({\bf 1}-z) .$$
Recall that Jordan $*$-isomorphisms preserve the disjointness for projections (see e.g. \cite[Proposition 2.14]{HSZ}).
Letting $B_i:= T(e_i) z= J(e_i )T(e_i) z $ and $A_i :=T(e_i) ({\bf 1}-z) = T(e_i)J(e_i)({\bf 1}-z) $, we complete the proof.
\end{proof}

\section{Uniqueness of symmetric structure}

Let $C_E$ be the symmetric operator ideal in $B(\cH)$ generated by a symmetric sequence space $E$.
We say that $C_E$ $\it{has \ a \ unique \ symmetric \ structure}$ if   $C_{E}$  isomorphic to some ideal   $C_F$  corresponding to   a symmetric sequence space  $F$ implies that    $E=F$ with equivalent norms.
At the International Conference on Banach Space Theory and its Applications at  Kent, Ohio (August 1979),
  Pe{\l}czy\'{n}ski posed the following question  concerning the symmetric structure of ideals of compact operators on the Hilbert space $\ell_2$
(see also \cite[Question (B)]{Arazy78} and \cite[Problem A]{Arazy81}
):
 \emph{Does   the    ideal $C_E$ of  compact    operators   corresponding to an arbitrary separable symmetric sequence space $E$  have   a  unique   symmetric  structure?}
For readers who are interested in this topic, we refer to \cite{HSS,Arazy78,Arazy81}.

We assume, in addition,  that
$\norm{e}_{E}=1$ for any atom  $e\in \cM$
if $\cM=B(\cH)$ equipped with the standard trace;   $\norm{{\bf 1}}_E=1 $ if $\cM$ is a type $II_1$-factor equipped with the unique faithful normal tracial state.
Here, we consider an analogue of Pe{\l}czy\'{n}ski's problem in the sense of isometric isomorphisms.
This assumption implies that if $\norm{\cdot}_E$ is proportional to $\norm{\cdot}_2$, then $\norm{\cdot}_E=\norm{\cdot}_2$.
Let $F(\cM,\tau)$ be  a symmetric operator space.
 If a  symmetric operator symmetric $E(\cM,\tau)$ isometric to   $F(\cM,\tau)$ implies that $E(\cM,\tau)$ coincides with $F(\cM,\tau)$,  then we say that $F(\cM,\tau)$ has a \emph{unique symmetric structure} up to an isometry.
The following corollary  extends results in  \cite{AZ,Kalton_R,R,Z77,Potepun} to the noncommutative setting.

\begin{cor}\label{unique}
Let $\cM=B(\cH)$ be  equipped with the standard trace $\tau$ or $\cM$ be a $II_1$-factor equipped with the unique faithful normal  tracial state $\tau$.
Let $E(\cM,\tau)$ and $F(\cM,\tau)$ be   symmetric operator spaces whose norms are order continuous and are not proportional to the norm of  $L_2(\cM,\tau)$.
Then, $T$ is a surjective isometry from $E(\cM,\tau)$ to $F(\cM,\tau)$ if and only if there exist a unitary   element  $u\in \cM $ and a trace-preserving Jordan $*$-isomorphism $J$ such that
\begin{align}\label{factor}T(x) =  u J(x), ~x\in \cM.
\end{align}
In particular,
any symmetric space $E(\cM,\tau)$ (including the case when $E(\cM,\tau)=L_2(\cM,\tau)$) has a unique symmetric structure up to an isometry.
\end{cor}
\begin{proof}
By Theorem \ref{th:iso}, it suffices to show that the Jordan $*$-isomorphism is trace-preserving.
Indeed,
when $\cM=B(\cH)$, this  follows from the fact that every $*$-automorphism on $B(\cH)$ is inner\cite[Corollary 5.42]{Douglas} (see also argument in \cite[p.75]{Sourour}).
 When $\cM$ is a $II_1$-factor, then the corollary follows from the  same argument in \cite[p. 118--119]{Sukochev}.

For the uniqueness of symmetric structure, we only need to show that if   $L_2(\cM,\tau)$ is isometric to $F(\cM,\tau)$ (when $\cM=B(\cH)$ or a $II_1$-factor), then $F(\cM,\tau)=L_2(\cM,\tau)$ (with the same norm).
Indeed, all other cases follow from \eqref{factor} immediately.

If there exists a surjective isometry $T:L_2(\cM,\tau)\to F(\cM,\tau)$, then
 $F(\cM,\tau)^\times$
is isometric to $L_2(\cM,\tau)$.
That is, $F(\cM,\tau)$ is isometric to $F(\cM,\tau)^\times$. In particular, both $F(\cM,\tau)$ and  $F(\cM,\tau)^\times$ have  the Fatou property and order continuous norms\cite[Theorem 45]{DP2014}.
Hence, $F(\cM,\tau) $ coincides with $F(\cM,\tau)^{\times \times}$ with the same norm\cite[Theorem 32]{DP2014}.
 If $\norm{\cdot}_{F^\times}$ is  proportional to   $\norm{\cdot}_2$, then the norm of  its  K\"{o}the dual  $F(\cM,\tau)^{\times \times}$ is also proportional to   $\norm{\cdot}_2$.
 By the assumption that $\norm{e}_{E}=1$ for any atom  $e\in \cM$
if $\cM=B(\cH)$ (or   $\norm{{\bf 1}}_E=1 $ if $\cM$ is a type $II_1$-factor), we obtain that  $F(\cM,\tau)$ coincide with  $L_2(\cM,\tau)$ and $\norm{\cdot}_F=\norm{\cdot}_2$.
Hence, we only need to consider the case when
 $\norm{\cdot}_{F^\times}$ is not proportional to   $\norm{\cdot}_2$.
Recalling that  $F(\cM,\tau)$ is isometric to $F(\cM,\tau)^\times$, by \eqref{factor}, $F(\cM,\tau)$  coincides with $F(\cM,\tau)^\times$ with $\norm{\cdot}_F=\norm{\cdot}_{F^\times }$, and therefore,
 for any $x\in F(\cM,\tau)$,
 by the definition of K\"{o}the dual,  we have
 $$\tau(xx^*)<\infty ,~x\in F(\cM,\tau), $$ i.e., $F(\cM,\tau)=F(\cM,\tau)^\times \subset L_2(\cM,\tau)$.
 On the other hand, by the definition of K\"{o}the dual, $F(\cM,\tau) \subset L_2(\cM,\tau)$ implies that $F(\cM,\tau)^\times \supset L_2(\cM,\tau)$.
 Hence, $F(\cM,\tau)=F(\cM,\tau)^\times =  L_2(\cM,\tau)$ (in the sense of sets).

Since $\norm{x}_2^2 =\tau(xx^*)\le \norm{x}_F\norm{x}_{F^\times  }=\norm{x}_F^2$, it follows that $\norm{x}_2 \le \norm{x}_F=\norm{x}_{F^\times }$.
Moreover,
$$\norm{x}_{F^\times  } = \sup_{ \norm{y}_F\le 1} \left|\tau(xy)\right | \le\sup_{ \norm{y}_2\le 1} \left| \tau(xy)\right| =\norm{x}_2.  $$
We obtain that
$\norm{x}_2=\norm{x}_F=\norm{x}_{F^\times  }$.
This completes the proof.
\end{proof}

\section{Final remarks}


\begin{rem}
Theorem \ref{th:iso} above covers  all existing  results of surjective isometries on  (complex) symmetric operator/function/sequence spaces in the literature.
 Indeed,
 \begin{enumerate}
   \item when  $\tau$ is finite, we have $T(x)=T(z)  J ( x)   +  J (x)T({\bf 1}-z ) $. This recovers  and extends  \cite[Theorem 3.1]{CMS}, \cite[Theorem 4.11]{JC} and \cite[Theorem 6]{SV}.
   \item  when  $\cM=L_\infty (\Omega, \Sigma,  \mu)$, where $ (\Omega, \Sigma,  \mu)$ is some $\sigma$-finite atomless measure space,  we have  $T(x) =  B J_1(x)   , ~x\in \cM,$ for   some   measurable function $B$ on $(\Omega, \Sigma,  \mu)$.
       In particular, Zaidenberg's results are  recovered (see \eqref{formf} and \cite{Z77,Zaidenberg}).
       \item when  $\cM= (\Gamma, \Sigma, \mu)$ is  a discrete measure space on a set $\Gamma$ with $\mu(\{\gamma\})=1$ for every $\gamma \in \Gamma$,
           we obtain that if $T$ is an isometry on a symmetric space $E(\Gamma)$ whose  norm is  order continuous norm and is not proportional to $\norm{\cdot}_2$, then
          $$(Tx)(\gamma)= A(\gamma )\cdot x(\sigma(\gamma)),~x\in E(\Gamma),~\gamma\in \Gamma,$$
            where $A(\gamma)$ are  unimodular scalars and $\sigma$ is a permutation on $\Gamma$.
           Indeed,
           by Theorem~\ref{th:iso}, there exists a
           surjective Jordan $*$-isomorphism
           $J$ on $\ell_\infty (\Gamma)$.
           By the  bijectivity and  disjointness-preserving property of $J$,
           we infer  that $J$ maps atoms onto atoms in $\ell_\infty (\Gamma)$.
           Hence, $J$ is generated by a permutation.
           This extends Arazy's description of isometries on  $E(\mathbb{N})$ \cite[Theorem 1]{Arazy85}.
   \item when $\cH$ is a (not necessarily separable) Hilbert space and $\cM=B(\cH)$, Theorem \ref{th:iso} recovers and extends Sourour's result\cite[Theorem 2]{Sourour}.
       \item when $\cM$ is a hyperfinite type $II$-factor,
       Theorem \ref{th:iso} extends \cite[Theorem 4.1]{Sukochev}, which was  established under the assumption that $\cH$ is separable.
   \item  when $E=F=L_p$,
   Theorem \ref{th:iso} extends and complements results in e.g. \cite{Sherman,Yeadon,Tam,Arazy,Katavolos2}.
   \item when $E$ and $F$ are Lorentz spaces, \cite[Theorem 6.1]{CMS} (see also \cite{KM,CHL,MS2,JC2}) is recovered.
   \item when $T$ is a positive isometry, several results are recovered and  extended (see e.g. \cite[Theorem 1]{Abramovich},  \cite[Theorem 3.1]{CMS}, \cite[Theorem 6]{SV},  \cite[Corollaries 5.4 and 5.5]{HSZ} and \cite{JC}).
       \item when $\cM$ is an atomic von Neumann algebra, Theorem \ref{th:iso} extends the main result in \cite{Medzhitov},
           which was established under the assumption that $\cM$ is a $\sigma$-finite von Neumann algebra.
 \end{enumerate}
\end{rem}

It is shown by Zaidenberg\cite{Zaidenberg} (see also \cite{FJ}) that   under certain conditions, every surjective isometry between two complex symmetric function spaces on a $\sigma$-finite atomless measure space   can be represented  in the form of \eqref{formf} (see \cite{Arazy85} for the case of symmetric sequence spaces).
In this respect,
the condition    imposed on the von Neumann algebras in Theorems \ref{th:her} and \ref{th:iso}
is very natural.
One may   expect that
Theorems \ref{th:her} and \ref{th:iso} (and \cite[Theorem 1]{Zaidenberg}) can be proved in the setting of more general von Neumann algebra, e.g., $\cM=\cM_1\oplus \cM_2$, where $\cM_1$ is an atomless von Neumann algebra (or $\cM_1$ is an atomic von Neumann algebra with all atoms having the same trace but the trace of an atom in $\cM_1$ is different from that in $\cM_2$) and $\cM_2$ is an atomic von Neumann algebra with all atoms having the same trace.
 See e.g. \cite[Definition 5.3.1]{FJ} or \cite{Bennett_S} for definition of symmetric function spaces on general measure spaces.
However,
the following simple example shows that for an atomic von Neumann algebra (or measure space) whose atoms have different measures, isometries may have different forms from \eqref{formf}.
 This demonstrates that the condition   imposed on the von Neumann algebras in Theorems \ref{th:her} and \ref{th:iso} is sharp.
For an arbitrary (not necessarily atomless or atomic with all atoms having the same trace) semifinite von Neumann algebras $\cM$,  it is interesting to characterize those symmetric
operator spaces $E(\cM,\tau)$ such that all
  isometries on $E(\cM,\tau)$ have elementary forms.

\begin{example}\label{exam}
Let $(\Omega,m)$ be a measure space consisting of two atoms $e_1$ and $e_2$.
Assume that $m(e_1)=1$ and $m(e_2)=2$.
Then, there exists a symmetric space $E(\Omega)$ which is not proportional to $L_2(\Omega)$
but it is isometric to the $2$-dimensional Hilbert space.
In particular, there exists a non-elementary isometry    and a  hermitian operatora   on $E(\Omega)$ which can not be written as   a  multiplication   of an element in $L_\infty (\Omega)$.
\end{example}
\begin{proof}
A non-trivial projection in $E(\Omega)$ must be $e_1$, $e_2$ or $e_1+e_2$,
where $e_1$ (resp. $e_2$ and $e_1+e_2$) denotes the indicator function on $e_1$ (resp. $e_2$ and $e_1+e_2$).

For an element $ae_1+be_2$, we
define a norm by
\begin{align}
\label{Edef}\norm{ae_1+be_2} _E := \sqrt{|a|^2 + 3 |b|^2}.
\end{align}
Indeed, this can be considered as the $L_2$-norm on a measure space having an atom of measure $1$ and the other of measure $\sqrt{3}$.
We claim that $\norm{\cdot}_E$ is symmetric.
Indeed, assume that
$ x:=a_1e_1 +a_2e_2\ge 0$ and $ y:=b_1e_1 +b_2e_2\ge 0 $ and $\mu(x)\le \mu(y)$.

If $b_1 \ge b_2$, then there are 2 possible cases:
\begin{enumerate}
  \item If $a_1\ge a_2$, then $b_1\ge a_1$ and $b_2\ge a_2$. In this case, we have   $\norm{x}_E\le \norm{y}_E$.
  \item $a_1\le a_2$. Since $m(e_1 )\le m(e_2)$ and $b_1\ge a_2$, it follows that $b_2\ge a_2\ge a_1$. Hence, $\norm{x}_E\le \norm{y}_E$.
\end{enumerate}

If $b_1\le b_2$, then there are 2 possible cases:
\begin{enumerate}
  \item If $a_1\ge a_2$, then $b_2\ge a_1\ge a_2 $ and $b_1\ge a_2$.
Note that $\norm{x}_E^2 = a_1^2 + 3 a_2^2 $ and $\norm{y}_E^2 = b_1^2 + 3 b_2^2 $.
We obtain that
$$\norm{y}_E^2-\norm{x}_E^2  = 3 b_2^2-a_1^2  +  b_1^2 -3 a_2^2 \ge  (3-1)b_2^2 -(3 -1) a_2^2\ge 0$$
  \item If $a_1\le a_2$, then $b_2\ge a_2$ and $b_1\ge a_1$.
Hence, $\norm{x}_E\le \norm{y}_E$.
\end{enumerate}

Consider the   matrix $T:=\left(
                              \begin{array}{cc}
                                -\frac{i}{\sqrt2} & \frac{\sqrt3 }{\sqrt2}\\
                                \frac{i}{\sqrt2 \cdot \sqrt3} & \frac{1}{\sqrt2} \\
                              \end{array}
                            \right)$.
                            That is,
$T(e_1)= -\frac{i}{\sqrt2}e_1 + \frac{i}{\sqrt2} \frac{1}{\sqrt3} e_2 $, $T(e_2 )= \sqrt{3} \frac{1}{\sqrt2}e_1 + \frac{1}{\sqrt2}e_2  $, $T(1)=( \frac{\sqrt3-i}{\sqrt2}  ,  \frac{i+\sqrt3 }{\sqrt2 \cdot \sqrt3} ) $.
 For any $a,b\in \mathbb{C}$, we have
\begin{align*}
\norm{T(ae_1+be_2)}_E^ 2    &= \norm{
 \left(
 \frac{-i a}{\sqrt{2} }+\frac{\sqrt{3}b}{\sqrt{2}}
 \right) e_1+
 \left(  \frac{ ia }{\sqrt{2}\sqrt{3}}  +\frac{b}{\sqrt {2}}
 \right)e_2
 }_E^ 2    \\
 & =
 \left|\frac{-i a}{\sqrt{2} }+\frac{\sqrt{3}b}{\sqrt{2}}   \right|^2 + 3 \left|  \frac{ ia }{\sqrt{2}\sqrt{3}}  +\frac{b}{\sqrt {2}}
 \right|^2 .
 \end{align*}
 By the Parallelogram law, we obtain that
 \begin{align*}
 \norm{T(ae_1+be_2)}_E^ 2
  =   |a|^2 +3 |b|^2 \stackrel{\eqref{Edef}}{=} \norm{ae_1+be_2}_E^ 2 .
 \end{align*}
 This implies that $T$ is an isometry on $E(\Omega)$.

 Assume that $T$ can be   written as an elementary form, that is, there exists a element in $E(\Omega)$ and a Jordan isomorphism on $L_\infty (\Omega)$ such that  $T=BJ$.
Since $J(e_1+e_2)=e_1+e_2$, it follows that
$$J(e_1)=e_1 ,~ J(e_2)=e_2$$
 or
 $$J(e_1)=e_2, ~J(e_2)=e_1.$$
However,
$T(e_1) =-\frac{i}{\sqrt2}e_1 + \frac{i}{\sqrt2} \frac{1}{\sqrt3} e_2 \ne BJ(e_1) $ for any $B\in E(\Omega)$.
Hence, $T$ cannot be   written in  an elementary form.


Consider a hermitian operator $T$ on $E(\Omega)$
 defined by
$$T:= \left(
        \begin{array}{cc}
          1 & i\cdot \sqrt{3} \\
          -\frac{i}{\sqrt{3}} & 1 \\
        \end{array}
      \right),
$$
i.e., $T(e_1) = e_1-\frac{i}{\sqrt{3}}e_2$ and $T(e_2)=i\cdot \sqrt{3} e_1 +e_2 $.
Assume that $T(x) =ax $ for some $a \in E(\Omega)$.
It follows that $e_1-\frac{i}{\sqrt{3}}e_2 = T(e_1)= ae_1 =\lambda e_1$ for some number $\lambda \in \mathbb{C}$, which is a contradiction. 
\end{proof}

\begin{remark}
Recall that Zaidenberg's description of isometries on complex symmetric function spaces
only requires
 that a symmetric function space has a Fatou norm, which is a slightly weaker assumption than  the requirement that this space has an order continuous  norm\cite[Theorem 5.3.5]{FJ}.
Throughout this paper, we always consider   symmetric   spaces having order continuous norms.
It will be interesting to
verify Theorem \ref{th:iso} under a slightly    weaker assumption that the symmetric operator spaces have Fatou norms only.
This problem is yet unsolved.
There are some partial answers  in this direction obtained in  \cite{AC,AC2,LPS,CLS} in the setting of $B(\cH)$.
\end{remark}

\begin{rem}

The   structure of (real or complex) symmetric sequence space (under the assumption that the spaces in question have the Fatou property, which is a stronger assumption than that of Fatou norm) has been discussed by
Braverman and Semenov\cite{BS74,BS75}, and   by Arazy\cite{Arazy85}.
Abramovich and Zaidenberg \cite[Theorem 1]{AZ} showed   $L_p[0,1]$, $1\le p<\infty$, has a unique structure up to an isometry.
The uniqueness of symmetric structure of separable complex symmetric function spaces on $[0,1]$ was obtained by Zaidenberg\cite{Z77}.
By a generalized Zaidenberg's theorem \cite[Theorem 1 and Proposition 3]{R} (see also \cite[Theorem 7.2]{Kalton_R}), the uniqueness of the symmetric structure of  separable  real symmetric  function spaces under some technical conditions  is obtained by Randrianantoanina\cite{R}.
\end{rem}
\begin{rem}
Note that for the case when $\cM$ is a $II_\infty$ factor, Corollary \ref{unique} may  fail because the Jordan $*$-isomorphism $J$ on $\cM$ may not be trace-preserving. This is an oversight in the proof of \cite[Theorem 4.1]{Sukochev}.
Indeed, letting $R_0=\otimes_{1\le n<\infty} \mathbb{M}_2$ be the hyperfinite $II_1$-factor equipped with the faithful normal  tracial state $\tau$, we consider the hyperfinite $II_\infty$-factor $\cM= B(\cH) \otimes R_0$ equipped with the trace ${\rm Tr}\otimes \tau$.
Let $\phi_1 $ be a    $*$-isomorphism from $ (R_0,\tau)  $ onto $(\otimes_{2\le n<\infty} \mathbb{M}_2,\tau)$,
 and $\phi_2$ be a natural
$*$-isomorphism from $B(\cH)\otimes {\bf 1}_{R_0}$ onto $B(\cH)\otimes \mathbb{M}_2\otimes {\bf 1}_{\otimes_{2\le n<\infty} \mathbb{M}_2}$.
Clearly, $\phi_1\otimes \phi_2$ is a   $*$-isomorphism on $\cM$ which does not preserves traces ${\rm Tr}\otimes \tau $.
 Indeed,
$\phi_2$ maps atoms in $B(\cH)\otimes {\bf 1}_{R_0} $ to atoms in $B(\cH)\otimes \mathbb{M}_2 \otimes {\bf 1}_{\otimes_{2\le n<\infty} \mathbb{M}_2}$.
Let $p\in B(\cH)$ be an atom.
In particular, ${\rm Tr}(p)=({\rm Tr}\otimes \tau  ) (p\otimes {\bf 1}_{R_0} )=1$ and $({\rm Tr}\otimes \tau  ) (\phi_1(p) \otimes   {\bf 1}_{\otimes_{2\le n<\infty} \mathbb{M}_2})=\frac12$.
This oversight in \cite[Theorem 4.1]{Sukochev} is rectified by Theorem \ref{th:iso} above.
It is natural to compare this result  with \cite[Theorem 5.3.5]{FJ} (see also \eqref{formf} and \cite[Main Theorem]{CHL}), where the set-isomorphism may not necessarily  preserve the measure.

\end{rem}

\appendix
\section{ }
 \label{appendix}

In this appendix, we  extend  \cite[Theorem 3.1]{Sukochev} and \cite[Lemma 2]{Sourour}
to the setting of arbitrary von Neumann algebras. Our technique is different from that used in \cite{Sukochev,Sourour}.
We are grateful to Dmitriy Zanin for providing us with a correction of our initial argument and allowing us to use it in this paper.

Let $\cM$ be a  von Neumann algebra.
Let $p$ be a projection in   $\cM$. We denote  by $z(p)$ the central support of $p$.
\begin{lem}\cite[Theorem 1.10.7]{Sakai71}\label{central support lemma} Let $p,q\in\mathcal{M}$ be projections  such that $$pyq=0, ~\forall y\in\mathcal{M}.$$
 We have $z(p)z(q)=0.$
\end{lem}
	

\begin{lem}\label{first computational lemma}	
Let $a,b,e,f\in\mathcal{M}$ be self-adjoint and such that
$$ey+yf=ayb,\quad \forall y\in\mathcal{M}.$$
We have
\begin{enumerate}
\item $[a,b]=0.$
\item $[b,[a,y]]=0$ for every $y\in\mathcal{M}.$
\end{enumerate}
\end{lem}
\begin{proof} Setting $y={\bf 1} ,$ we obtain $e+f=ab.$ Taking adjoint, we obtain $e+f=ba.$ Thus, $ab=ba.$ This proves the first assertion.
	
Substituting $f=ab-e,$ we obtain
$$[e,y]=[a,y]b,\quad y\in\mathcal{M}.$$
Taking adjoints, we obtain
$$[e,y]=b[a,y],\quad y\in\mathcal{M}.$$
Comparing the right hand sides, we establish the second assertion.
\end{proof}

\begin{lem}\label{second computational lemma} Let $a,b\in\mathcal{M}$ be commuting self-adjoint elements such that
$$[b,[a,y]]=0,\quad \forall y\in\mathcal{M}.$$
We have
$$[p,[q,y]]=0,\quad \forall y\in\mathcal{M}$$
for all spectral projections $p$ and $q$ of $a$ and $b,$ respectively.
\end{lem}
\begin{proof} Note that for all $y\in\mathcal{M}$, we have
$$[b^n,[a,y]]=
[b, [b^{n-1},[a,y]]] + [ [b,[a,y]], b^{n-1} ]=
[b, [b^{n-1},[a,y]]] =\cdots =  [ b, [b,\cdots [b,[a,y]]\cdots ] ]
=0. $$
By linearity,
$$[P(b),[a,y]]=0,\quad y\in\mathcal{M},$$
for every polynomial $P.$ Since polynomials are norm-dense in the algebra of continuous functions, it follows that
$$[x,[a,y]]=0,\quad y\in\mathcal{M},$$
for every $x$ in the $C^{\ast}$-algebra generated by $b.$ By weak continuity of our equation,
$$[x,[a,y]]=0,\quad y\in\mathcal{M},$$
for every $x$ in the von Neumann algebra generated by $b.$ In particular,
$$[q,[a,y]]=0,\quad y\in\mathcal{M}.$$
Using the Leibniz rule
$$ [q,[a,y]]+[a,[y,q]]+[y,[q,a]]=0   $$
and taking into account that $[q,a]=0,$   we have
$$[a,[q,y]]=0,\quad y\in\mathcal{M}.$$
Repeating the argument in the first paragraph, we complete the proof.
\end{proof}

\begin{lem}\label{main projection lemma} If $p,q\in\mathcal{M}$ are commuting projections such that
\begin{align}\label{A2}
[p,[q,y]]=0,\quad \forall y\in\mathcal{M},
\end{align}
then
$$z(p)z(q)z({\bf 1}-p)z({\bf 1}-q)=0.$$
\end{lem}
\begin{proof}
Denote
$$z':=z(p)z(q)z({\bf 1}-p)z({\bf 1} -q).$$
Assume by contradiction that $z'\neq 0.$ By \eqref{A2}, we have
$$[pz',[qz',y]]=0,\quad y\in  \mathcal{M}_{z'}$$
and
$$z(pz')z(qz')z(z'-pz')z(z'-qz')\stackrel{\tiny \mbox{\cite[Proposition 5.5.3]{KR}}}{=} z'\cdot z(p)z(q)z({\bf1}-p)z({\bf 1} -q)=z'.$$
Hence, by passing to the reduced von Neumann algebra $\cM_{z'}$,  we may assume without loss of generality that $z'={\bf 1} .$ In other words,
\begin{align}\label{assume}
z(p)=z(q)=z({\bf 1}-p)=z({\bf 1}-q)={\bf 1} .
\end{align}

Obviously, the assumption \eqref{A2} is equivalent to
\begin{equation}\label{mpl eq0}
pqy+ypq=pyq+qyp.
\end{equation}

Replacing $y$ in \eqref{mpl eq0} with $({\bf 1}-q)y({\bf 1} -p),$ we obtain
$$0+0=p({\bf 1}-q)y({\bf 1}-p)q+0,\quad y\in\mathcal{M}.$$
By Lemma \ref{central support lemma}, we have
\begin{align}\label{A3'}
z(p({\bf 1} -q))\cdot z(({\bf 1}-p)q)=0.
\end{align}

Let $w_1:=z(p({\bf 1}-q)),$ $w_2=:z(({\bf 1} -p)q).$ We have $w_1w_2=0.$ Set $w_3:={\bf 1} -w_1-w_2.$
By \eqref{A2}, we have
\begin{align}\label{A3}
[pw_1,[qw_1,y]]=0,\quad y\in w_1\mathcal{M},
\end{align}
$$[pw_2,[qw_2,y]]=0,\quad y\in w_2\mathcal{M},$$
$$[pw_3,[qw_3,y]]=0,\quad y\in w_3\mathcal{M}.$$

{\bf Step 1:} We claim
 that
\begin{align}\label{A4'}qw_1\leq pw_1,~ pw_2\leq qw_2 \mbox{ and }pw_3=qw_3.
\end{align}
Note that
$$z((w_1-pw_1)\cdot qw_1)\stackrel{\tiny \mbox{\cite[Proposition 5.5.3]{KR}}}{=}w_1\cdot z(({\bf 1} -p)q)=w_1\cdot w_2\stackrel{\eqref{A3'}}{=}0.$$
Hence,
\begin{align}\label{A4}
(w_1-pw_1)\cdot qw_1=0.
\end{align}
In other words, $qw_1\leq pw_1.$

Similarly,
$$z(pw_2\cdot (w_2-qw_2))\stackrel{\tiny \mbox{\cite[Proposition 5.5.3]{KR}}}{=} z(p({\bf 1} -q))\cdot w_2=w_1\cdot w_2\stackrel{\eqref{A3'}}{=} 0.$$
Hence,
$$pw_2\cdot (w_2-qw_2)=0.$$
In other words, $pw_2\leq qw_2.$

Arguing similarly, we have $pw_3\leq qw_3$ and $qw_3\leq pw_3.$
This completes the proof of \eqref{A4'}.

{\bf Step 2:} We claim that
\begin{align}
\label{A5}w_1 =0,\quad w_2  =0,\quad w_3  =0.
\end{align}
We only prove the first equality. Proofs of the other $2$ are similar.

By \eqref{A3}, we  have
$$[pw_1,[qw_1,(w_1-pw_1)y]]=0,\quad y\in \mathcal{M}_{w_1}.$$
Since
\begin{align}\label{w1pw1}
 (w_1-pw_1) \cdot qw_1\stackrel{\eqref{A4}}{=} 0,
 \end{align}
it follows that for all  $y\in  \mathcal{M}_{w_1} $, we have
$$0= [pw_1,[qw_1,(w_1-pw_1)y]]= [pw_1,qw_1 (w_1-pw_1)y - (w_1-pw_1)y qw_1   ]   \stackrel{\eqref{w1pw1}}{=}- [pw_1,  (w_1-pw_1)yqw_1 ] .  $$
Since $pw_1\cdot (w_1-pw_1)=0$ and since $qw_1\stackrel{\eqref{A4'}}{\leq} pw_1,$ it follows that
$$(w_1-pw_1)yqw_1=0,\quad y\in   \mathcal{M}_{w_1} .$$
By Lemma \ref{central support lemma}, we have
$$z(w_1-pw_1)\cdot z(qw_1)=0.$$
In other words,
$$w_1 \stackrel{\eqref{assume}}{=}    z({\bf 1}-p)z(q)w_1\stackrel{\tiny \mbox{\cite[Proposition 5.5.3]{KR}}}{=} z(w_1-pw_1)\cdot z(qw_1)=0  .$$
This proves the first equality of \eqref{A5}.


Finally, ${\bf 1}=w_1+w_2+w_3 =0$, which is impossible.
Hence,
$z=0$. This completes the proof.
\end{proof}

\begin{lem} \label{la5}If $a,b\in\mathcal{M}$ are commuting elements such that
$$[b,[a,y]]=0,\quad \forall y\in\mathcal{M},$$
then there exists a  central projection $z$ such that both $a({\bf 1} -z)$ and $bz$ are central.
\end{lem}
\begin{proof}
Since   $a$ commutes with $b$,
it follows   from Lemmas \ref{second computational lemma} and  \ref{main projection lemma} that
$$z(p)z(q)z({\bf 1}-p)z({\bf 1}-q)=0$$
for arbitrary spectral projection $p$ (respectively, $q$) of $a$ (respectively, $b$).

Denote, for brevity, $z_q=z(q)z({\bf 1} -q).$ We have
$$z(pz_q)\cdot z(z_q-pz_q)\stackrel{\tiny \mbox{\cite[Proposition 5.5.3]{KR}}}{=} z(p)z({\bf 1}-p)z_q=z(p)z({\bf 1}-p)z(q)z({\bf 1}-q)=0.$$
Therefore,
$$0=z(pz_q) z( z_q-pz_q) \ge z(pz_q) ( z_q-pz_q) \ge  ( z_q-pz_q)  z(pz_q) ( z_q-pz_q)  \ge  ( z_q-pz_q)   pz_q   ( z_q-pz_q)    =0 .$$
This implies that
$$ pz_q \ge  z(pz_q) \ge pz_q ,$$
that is,
\begin{align}\label{centralz}
pz_q =z(pz_q)\in Z(\cM).
\end{align}
Thus, $pz(q)z({\bf 1}-q)=pz_q$ is a  central projection. By the Spectral Theorem, $az(q)z({\bf 1} -q)\in Z(\cM)$.

Define $z'\in \cP(Z(\cM))$ by
$${\bf 1}-z'=\bigvee_q z(q)z({\bf 1}-q),$$
where the supremum is taken over all spectral projections $q$ of  $b$.
We have $a({\bf 1}-z')= a\cdot  \bigvee_q z(q)z({\bf 1}-q) \in Z(\cM) $. On the other hand,   we have (see e.g. \cite[Chapter V, Proposition 1.1]{Tak})
$$z'=\bigwedge ({\bf 1}-z(q)z({\bf 1}-q)).$$
In particular,
\begin{align}\label{z'leq}
z'\leq {\bf 1}-z(q)z({\bf 1}-q)
\end{align}
for every $q.$ Thus,
$$z'-z(z'q)z(z'-z'q)\stackrel{\tiny \mbox{\cite[Proposition 5.5.3]{KR}}}{=} z'\cdot ({\bf 1}-z(q)z({\bf 1}-q))\stackrel{\eqref{z'leq}}{=}z'$$
for every $q$. That is,
$$z(z'q)z(z'-z'q)=0$$
for every $q.$ Hence, $z'q\in Z(\cM)$ for every spectral projection $q$ of $b$ (see e.g. the proof for \eqref{centralz}).  By the Spectral Theorem, $bz'\in Z(\cM)$.
\end{proof}

The following theorem is an immediate consequence of
Lemmas \ref{first computational lemma}	 and \ref{la5}.
  It should be compared with   \cite[Theorem 3.1]{Sukochev} and \cite[Lemma 2]{Sourour}.
\begin{theorem} \label{central decomposition th}
Let $a,b,e,f\in\mathcal{M}$ be self-adjoint and such that
\begin{align}\label{ey+yf}
ey+yf=ayb,\quad \forall y\in\mathcal{M}.
\end{align}
Then there exists  a  central projection $z$ such that $a({\bf 1} -z), e({\bf 1} -z )$ and $bz,fz$ are central.
\end{theorem}
\begin{proof}
By Lemmas \ref{first computational lemma}	 and \ref{la5}, we obtain that there exists $z\in \cP(Z(\cM))$ such that
$a({\bf 1} -z), bz\in Z(\cM)$.
Replacing   $y $ with $ z$ in \eqref{ey+yf}, we obtain that
\begin{align*}
e z+ fz= ab z,\quad \forall y\in\mathcal{M}_z.
\end{align*}
 We have $fz= abz-ez$.
 Hence,
 \begin{align*}
  yfz \stackrel{\eqref{ey+yf}}{=} (ab z-ez)y =fz y,\quad \forall y\in\mathcal{M}_z.
\end{align*}
 This implies that $fz\in Z(\cM)$.
 The same argument  shows that $e({\bf 1} -z )\in Z(\cM)$.
 \end{proof}

\bibliographystyle{amsalpha}

\end{document}